\documentclass[12pt]{amsart}               
\usepackage{a4wide}  
\usepackage[autostyle]{csquotes}
\usepackage{mathtools}
\usepackage{comment}
\usepackage{ytableau}
\usepackage{tikz}
\usepackage[backend=biber]{biblatex} 
\usepackage[utf8]{inputenc} 
\DeclareFieldFormat{doi}{\href{https://dx.doi.org/#1}{\textsc{doi}}}
\DeclareFieldFormat{url}{\href{#1}{\textsc{url}}}
\renewbibmacro*{doi+eprint+url}{
  \printfield{doi}
  \newunit\newblock
  \iftoggle{bbx:eprint}{
    \usebibmacro{eprint}
  }{}
  \newunit\newblock
  \iffieldundef{doi}{
    \usebibmacro{url+urldate}}
  {}
}
\renewbibmacro{in:}{}
\AtEveryBibitem{
  \ifentrytype{book}
  {\clearfield{pages}}
}

\usepackage[hidelinks,colorlinks=true,linkcolor=blue,citecolor=magenta]{hyperref}

\makeatletter
\@namedef{subjclassname@2020}{\textup{2020} Mathematics Subject Classification}
\makeatother

\allowdisplaybreaks[1] 
\title[Diagonal operators and $q$-Whittaker functions]{Diagonal operators, $q$-Whittaker functions\\ and rook theory}     
\def\PP{\mathbb{P}}
\def\Fq{\mathbb{F}_q}
\def\inv{\mathrm{inv}}
\def\BB{\mathcal{B}}
\def\AA{\mathcal{A}}
\def\TT{\mathcal{T}} 
\def\SS{\mathcal{S}}
\def\salpha{\mathcal{S}_\alpha}
\def\ssyt{\rm SSYT}
\def\syt{\rm SYT}
\def\bb{\mathsf{b}}
\def\diag{\mathrm{diag}}
\def\ninv{\mathrm{ninv}}
\def\MM{{\rm M}}

\def\tabnmu{{\rm Tab}_{\subseteq [n]}(\mu)}
\def\tabnmualpha{{\rm Tab}_{\subseteq [n]}(\mu,\alpha)}

\def\sstabmua{{\rm SSTab(\mu,\alpha)}}
\def\bx{{\bf x}}
\theoremstyle{definition}
\newtheorem{theorem}{Theorem}[section]
\newtheorem{corollary}[theorem]{Corollary}
\newtheorem{example}[theorem]{Example}
\newtheorem{lemma}[theorem]{Lemma}  
\newtheorem{problem}[theorem]{Problem}

\newtheorem{proposition}[theorem]{Proposition}
\newtheorem{definition}[theorem]{Definition}
\newtheorem{remark}[theorem]{Remark}

\author{Samrith Ram} 
\address{Indraprastha Institute of Information Technology Delhi, New Delhi, India.}
\email{samrith@iiitd.ac.in}
\author[Michael J.\ Schlosser]{Michael J.\ Schlosser} 
\address{Fakult\"at f\"ur Mathematik, Universit\"at Wien,
Oskar-Morgenstern-Platz~1, A-1090 Vienna, Austria}
\email{michael.schlosser@univie.ac.at} 

\begin{document}
\begin{abstract}
  We discuss the problem posed by Bender, Coley, Robbins and Rumsey of enumerating the number of subspaces which have a given profile with respect to a linear operator over the finite field $\Fq$. We solve this problem in the case where the operator is diagonalizable. The solution leads us to a new class of polynomials $b_{\mu\nu}(q)$ indexed by pairs of integer partitions. These polynomials have several interesting specializations and can be expressed as positive sums over semistandard tableaux. We present a new correspondence between set partitions and semistandard tableaux. A close analysis of this correspondence reveals the existence of several new set partition statistics which generate the polynomials $b_{\mu\nu}(q)$; each such statistic arises from a Mahonian statistic on multiset permutations. The polynomials $b_{\mu\nu}(q)$ are also given a description in terms of coefficients in the monomial expansion of $q$-Whittaker symmetric functions which are specializations of Macdonald polynomials. We express the Touchard--Riordan generating polynomial for chord diagrams by number of crossings in terms of $q$-Whittaker functions. We also introduce a class of $q$-Stirling numbers defined in terms of the polynomials $b_{\mu\nu}(q)$ and present connections with $q$-rook theory in the spirit of Garsia and Remmel.         
\end{abstract} 
\maketitle
\tableofcontents
\section{Introduction}     
Let $\Fq$ denote the finite field with $q$ elements where $q$ is a prime power. Let $n$ be a positive integer and denote by $\MM_n(\Fq)$ the space of $n\times n$ matrices over $\Fq$.
\begin{definition}\label{def:profile}
Given a matrix $\Delta \in \MM_n(\Fq)$, a subspace $W$ of $\Fq^n$ has $\Delta$-profile $\mu=(\mu_1,\mu_2,\ldots)$ if 
\begin{align*}
  \dim (W+\Delta W+\cdots +\Delta^{j-1}W)=\mu_1+\cdots+\mu_j \quad (j\geq 1). 
\end{align*}
\end{definition}
It can be shown that the $\Delta$-profile of each subspace is a weakly decreasing sequence of nonnegative integers with finite sum \cite[p.\ 2]{MR1141317}; thus it can be viewed as an integer partition. The $\Delta$-profile of a subspace is also referred to as `dimension sequence' \cite{MR1141317}. Let $\sigma(\mu,\Delta)$ denote the number of subspaces of $\Fq^n$ with $\Delta$-profile $\mu$. We are primarily interested in the following problem. 
\begin{problem}\label{prob:main}
Give an explicit formula for $\sigma(\mu,\Delta)$ for arbitrary $\mu$ and $\Delta$.
\end{problem}
This problem was originally posed by Bender, Coley, Robbins and Rumsey \cite[p.\ 2]{MR1141317} who gave elegant product formulas in the cases where $\Delta$ has irreducible characteristic polynomial or is regular nilpotent (nilpotent with one-dimensional null space). When $\Delta$ has irreducible characteristic polynomial, they proved that 
\begin{equation}\label{eq:irreducibleprofiles}
  \sigma(\mu,\Delta) = \frac{q^n-1}{q^{\mu_1}-1}\prod_{i\geq 2}q^{\mu_i^2-\mu_i}{\mu_{i-1} \brack \mu_i}_q.
\end{equation}
Similarly, when $\Delta$ is regular nilpotent,
\begin{displaymath}
  \sigma(\mu,\Delta) = \prod_{i\geq 2}q^{\mu_i^2}{\mu_{i-1} \brack \mu_i}_q.
\end{displaymath}

Here ${n \brack k}_q$ denotes a $q$-binomial coefficient, given by
\begin{align*}
  {n \brack k}_q=\prod_{i=0}^{k-1} \frac{q^n-q^i}{q^k-q^i}.
\end{align*}
Unaware of the work of the authors of \cite{MR1141317}, proofs of Eq.\ \eqref{eq:irreducibleprofiles} in the special case where $\mu$ is a rectangular partition of $n$ appear later on in the work of Chen and Tseng \cite{MR3093853} and in \cite{MR4263652}. In this special case, the problem of determining $\sigma(\mu,\Delta)$ has interesting connections with group theory and finite projective geometry and is related to the splitting subspace problem posed by Niederreiter \cite[p.\ 11]{MR1334623} in the context of pseudorandom number generation. We refer to \cite{MR2831705,MR2961399} for details on this topic. It is easy to see that $\sigma(\mu,\Delta)$ depends only on the conjugacy class of $\Delta$ since $W$ has $\Delta$-profile $\mu$ if and only if $PW$ has $P\Delta P^{-1}$-profile $\mu$ for each linear isomorphism $P$ of $\Fq^n$. More generally, it can be shown that $\sigma(\mu,\Delta)$ depends only on the similarity class type (in the sense of Green \cite{MR72878}) of $\Delta$ \cite[Cor.\ 4.7]{pr}. An answer to Problem \ref{prob:main} in the case where the invariant factors of $\Delta$ satisfy certain degree constraints and $\mu$ has equal parts appears in \cite{MR4448290}. The case where $\mu$ is a partition of the ambient dimension with exactly two parts was resolved in \cite{prasad2023enumeration}. A formula is also known \cite{pr} in the case where $\Delta$ is a regular diagonal matrix (the diagonal entries are distinct). The regular diagonal case has interesting connections with several classical combinatorial objects as we outline below.   
\begin{theorem}\label{thm:regularprofiles}\cite[Thm.\ 4.8]{pr}
    For each integer partition $\mu$, there exists a polynomial $b_{\mu}(t)$ with nonnegative integer coefficients such that 
  \begin{align*}
    \sigma(\mu,\Delta)={n \choose |\mu|}(q-1)^{\sum_{j\geq 2} \mu_j}q^{\sum_{j\geq 2}{\mu_j \choose 2}}b_{\mu}(q),
  \end{align*}
for every prime power $q$ and every regular diagonal matrix $\Delta\in \MM_n(\Fq).$
\end{theorem}
The polynomials $b_\mu(t)$ above have several interesting specializations. $b_\mu(1)$ equals the number of set partitions of an $n$-element set of shape $\mu'$ (the conjugate of $\mu$) while $b_\mu(0)$ equals the number of standard Young tableaux of shape $\mu'$. In addition, $b_\mu(-1)$ equals the number of standard shifted tableaux of shape $\mu'$ when $\mu'$ has distinct parts. In the case where $\mu=(m,m)$, we have $b_{\mu}(t)=T_m(t)$, the Touchard generating polynomial for chord diagrams on $2m$ points by number of crossings. 
 The polynomials $b_\mu(t)$ can be obtained by summing a polynomial statistic over standard tableaux of shape $\mu'$ and also via a new statistic on set partitions called the interlacing number. Several classical objects like the Stirling numbers and their $q$-analogs (as defined by Carlitz \cite{MR1501675}) as well as the Bell numbers can be expressed in terms of the polynomials $b_\mu(t)$ and their specializations. The reader is referred to \cite{pr} for further details.        

In this paper, we investigate Problem \ref{prob:main} in the case where $\Delta$ is an arbitrary diagonal matrix. Before we state our main results, it will be convenient to introduce some notation. Throughout this article $[n]$ denotes the set of the first $n$ positive integers while $\Pi_n$ denotes the collection of all set partitions of $[n]$. By convention we write set partitions in standard form: the elements in each block are listed in increasing order while the blocks are written in increasing order of their least elements. The \emph{shape} of a set partition is the integer partition obtained by arranging the cardinalities of its blocks in weakly decreasing order. A \emph{weak composition} of an integer $n$ is a sequence $\alpha=(\alpha_1,\ldots,\alpha_k)$ of nonnegative integers with $|\alpha|:=\sum_{i=1}^k \alpha_i=n$. In addition, if each $\alpha_i>0$, we say that $\alpha$ is a composition of $n$. 
\begin{definition}\label{def:canonical}
  To each composition $\alpha=(\alpha_1,\ldots,\alpha_k)$ of $n$, we associate a canonical set partition $\mathcal{A}_\alpha\in \Pi_n$ by setting $\mathcal{A}_\alpha=\cup_{i=1}^kA_i$ where $A_i$ $(1\leq i\leq k)$ is the set of $\alpha_i$ consecutive integers given by $$A_i=\{\alpha_1+\cdots+\alpha_{i-1}+1,\ldots,\alpha_1+\cdots+\alpha_{i}\}.$$
  \end{definition}
For instance, when $\alpha=(2,3,1)$, we have $\AA_\alpha=\{\{1,2\},\{3,4,5\},\{6\}\}$ which we abbreviate to $12|345|6$. The \emph{type} of a diagonal matrix is the integer partition obtained by listing the multiplicities of its diagonal entries in weakly decreasing order. We are now ready to state the main results of this paper. We obtain the following answer to Problem~\ref{prob:main} in the diagonal case, thereby generalizing Theorem \ref{thm:regularprofiles}.

\begin{theorem}\label{thm:intro}
For each pair $\mu,\nu$ of integer partitions, there exists a polynomial $b_{\mu\nu}(t)$ with nonnegative integer coefficients such that 
  \begin{align*}
    \sigma(\mu,\Delta)=(q-1)^{\sum_{j\geq 2} \mu_j}q^{\sum_{j\geq 2}{\mu_j \choose 2}}b_{\mu\nu}(q),
  \end{align*}
for every prime power $q$ and every diagonal matrix $\Delta\in \MM_n(\Fq)$ of type $\nu$.
\end{theorem}
We note that $b_{\mu\nu}(t)={n \choose |\mu|} b_{\mu}(t)$ when $\nu$ has all parts equal to 1. In the case where $\mu$ is a partition of $n$, the polynomials $b_{\mu\nu}(t)$ have the following specializations: $b_{\mu\nu}(0)$ is the Kostka number $K_{\mu\nu}$, defined as the number of semistandard tableaux of shape $\mu$ and content $\nu$; $b_{\mu\nu}(1)$ equals the number of set partitions $\AA\in \Pi_n$ of shape $\mu'$  such that no block of $\AA$ intersects any block of the canonical set partition $\AA_\nu$ in more than one element (the meet of $\AA$ and $\AA_\nu$ in the lattice $\Pi_n$ is the minimal element).

Motivated by the above specializations, we prove that the polynomials $b_{\mu\nu}(t)$ can be written as positive sums over semistandard tableaux of shape $\mu$ and content $\nu$. We also prove the existence of several new set partition statistics that generate the polynomials $b_{\mu\nu}(t)$. More precisely, we show that to each Mahonian statistic (in the sense of Definition~\ref{def:mahonian}) on multiset permutations, one can associate a set partition statistic that generates the polynomials $b_{\mu\nu}(t)$ (Theorem \ref{thm:bviasetstat}). The expressions for $b_{\mu\nu}(t)$ in terms of semistandard tableaux and set partitions lead us to an intriguing elementary correspondence (see Section~\ref{sec:semistandardandset}) between semistandard tableaux and set partitions. To the best of our knowledge, this correspondence appears to be new. Theorem \ref{thm:intro} has proved to be a crucial ingredient in the resolution of the general case of Problem \ref{prob:main} (see \cite{ram2023subspace}).

In Section \ref{sec:symmetric}, we explore connections between the polynomials $b_{\mu\nu}(t)$ and the theory of symmetric functions. In this context, and throughout this paper, it will be convenient to treat $q$ as a formal variable, which we will often specialize to a prime power. The $q$-Whittaker functions $W_\lambda({\bf x};q)$ were defined by Gerasimov, Lebedev and Oblezin \cite{MR2575477} and may be viewed as joint eigenfunctions of $q$-deformed Toda chain Hamiltonians with support in the positive Weyl chamber (see Etingof \cite{MR1729357} or Ruijsenaars \cite{MR1090424}). They also arise as specializations of a more general and well-studied class of symmetric functions, the Macdonald polynomials $P_\lambda({\bf x};q,t)$, as $W_\lambda({\bf x};q)=P_\lambda({\bf x};q,0)$ (Macdonald \cite{macdonald1988new}). We prove the following relation between the polynomials $b_{\mu\nu}(q)$ and the $q$-Whittaker functions.        
\begin{theorem}\label{thm:whittakerconn}  
If $W_{\mu}=\sum_{\nu}a_{\mu\nu}(q)m_\nu$ denotes the monomial expansion of the $q$-Whittaker function, then
  \begin{align*}
 b_{\mu\nu}(q)   =\frac{\prod_{i\geq 1}[\nu_i]_q!}{\prod_{i\geq 1}[\mu_i-\mu_{i+1}]_q!}  a_{\mu\nu}(q),
  \end{align*}
  where $[n]_q!$ denotes the product $\prod_{i=1}^n(1+q+\cdots+q^{i-1})$.
\end{theorem}
Let $M_{\lambda\nu}$ denote the number of binary integer matrices with row sums $\lambda$ and column sums $\nu$. It is well-known that the numbers $M_{\lambda\nu}$ occur as coefficients in the monomial expansion of elementary symmetric functions (Stanley \cite[Prop.\ 7.4.1]{MR1676282}). As an application of Theorem~\ref{thm:whittakerconn}, we give an efficient nonrecursive formula for computing $M_{\lambda\nu}$ when the Kostka number $K_{\lambda'\nu}$ is small (Corollary~\ref{cor:binmatviassyt}).  

Let $T_m(q)$ denote the generating polynomial for chord diagrams on $2m$ points by their number of crossings. There is a beautiful analytic formula for these polynomials, referred to as the Touchard--Riordan formula (see Eq.\ \eqref{eq:touchard-riordan}).  The study of subspace profiles for regular diagonal matrices has recently led to a new proof of this formula \cite{MR4555237}. Using results in Section \ref{sec:symmetric}, we prove (Theorem \ref{thm:touchardasscalar}) that these polynomials admit a compact representation in terms of $q$-Whittaker functions:  
\begin{align*}
    T_m(q)=\frac{1}{[m]_q!} \langle W_{(m,m)},h_1^{2m} \rangle.
\end{align*}
Here $\langle\cdot,\cdot \rangle$ denotes the Hall scalar product \cite[p.\ 63]{MR1354144} on the ring of symmetric functions while $h_\lambda$ denotes the complete homogeneous symmetric function indexed by the partition $\lambda$.

In Section \ref{sec:stirling}, we define a new class of $q$-Stirling numbers $S_q(n,m;\nu)$ indexed by integer partitions $\nu$ of $n$. These numbers are defined in terms of the polynomials $b_{\mu\nu}(q)$ by
\begin{align*}
    S_q(n,m;\nu):=\sum_{\substack{\mu\vdash n\\\mu_1=m}}q^{\sum_{j\geq 2}{\mu_j \choose 2}}b_{\mu\nu}(q),
\end{align*}
where the sum is taken over all partitions $\mu$ of $n$ with first part $m$. In the case where $\nu$ has all parts equal to 1, $S_q(n,m;\nu)$ coincides with the $q$-Stirling numbers of the second kind defined by Carlitz. In Section \ref{sec:rooktheory} we discuss connections with $q$-rook theory introduced by Garsia and Remmel \cite{MR834272}. We prove that the $q$-Stirling numbers $S_q(n,m;\nu)$ above arise naturally as $q$-rook numbers of certain truncated staircase boards.

The paper is organized as follows. In Section \ref{sec:semistandardandset} we discuss the elementary correspondence between semistandard tableaux and set partitions. Section \ref{sec:setstat} consists of a description of the various set partition statistics which arise from Mahonian statistics and which generate the polynomials $b_{\mu\nu}(q)$. In Section \ref{sec:diagonal} we use a cell decomposition of the Grassmanian to prove that the polynomials $b_{\mu\nu}(q)$ arise naturally when counting subspaces with profile $\mu$ with respect to a diagonal operator of type $\nu$. The connection with $q$-Whittaker symmetric functions and the Touchard--Riordan formula is explored in Section \ref{sec:symmetric}. In Sections \ref{sec:stirling} and \ref{sec:rooktheory}, we define a class of $q$-Stirling numbers indexed by integer partitions and present connections with $q$-rook theory.

\section{Semistandard tableau associated to a set partition}
\label{sec:semistandardandset}
In this section we discuss the correspondence between semistandard tableaux and set partitions stated in the introduction. This correspondence is of an elementary nature and can be explained without reference to the subspace counting problem discussed in the introduction. We begin with some notation.

A partition of an integer $n$ is a weakly decreasing sequence $\mu=(\mu_1,\mu_2,\ldots)$ of nonnegative integers with sum $n$. If $\mu$ is a partition of $n$, we write $\mu\vdash n$. We also write $|\mu|$ for the sum $\sum_{i\geq 1}\mu_i$. The positive $\mu_i$ $(i\geq 1)$ are referred to as the parts of $\mu$. The number of parts of $\mu$ is called the length of $\mu$, denoted $\ell(\mu)$. It is customary to omit trailing zeroes when writing partitions. For instance, the partition $(3,1,1,0,0,\ldots)$ of $5$ is considered equivalent to $(3,1,1)$. The Young diagram of a partition $\mu$ is an array of cells arranged in rows such that the $i$th row has precisely $\mu_i$ cells for $1\leq i\leq \ell(\mu)$. A tableau of shape $\mu$ is a filling of the Young diagram of $\mu$ with positive integers. A tableau of shape $\mu$ is said to be \emph{standard} if its entries are precisely the first $|\mu|$ positive integers and the entries in each row and column are strictly increasing from left to right and top to bottom (see Figure \ref{fig:standssyt}). Given a weak composition $\alpha=(\alpha_1,\alpha_2,\ldots)$, a \emph{semistandard} tableau of shape $\mu$ and content $\alpha$ is tableau of shape $\mu$ in which the rows are weakly increasing from left to right, the columns are strictly increasing from top to bottom and the entry $i$ appears precisely $\alpha_i$ times (see Figure \ref{fig:standssyt}). The set of semistandard tableaux of shape $\mu$ and content $\alpha$ is denoted ${\rm SSYT}(\mu,\alpha)$. 
\ytableausetup{smalltableaux}

  \begin{figure}[h]
    \centering
     \begin{ytableau}
       1 & 2 & 4 & 5& 9 \\
       3 & 6 & 7 & \none &\none\\
       8 &\none&\none&\none &\none
     \end{ytableau}\qquad
  \begin{ytableau}
       1 & 1 & 1 & 2& 3 \\
       2 &3 & 3 & \none &\none\\
       3 &\none&\none&\none &\none
\end{ytableau}.
    \caption{Left: A standard tableau of shape $(5,3,1)$. Right: A semistandard tableau of shape $(5,3,1)$ and content $(3,2,4).$}
    \label{fig:standssyt}
  \end{figure}

  \begin{definition}
Two set partitions $\AA$ and $\mathcal{B}$ are \emph{minimally intersecting} if each block of $\AA$ intersects each block of $\mathcal{B}$ in at most one element.     
  \end{definition}
The collection of set partitions of $[n]$ of a given shape $\lambda$ is denoted $\Pi_n(\lambda)$. Let $\lambda$ be a partition of $n$ and suppose $\alpha$ is a composition of $n$. Let $\AA_\alpha$ denote the canonical set partition corresponding to $\alpha$ as in Definition \ref{def:canonical}. Denote by $\Pi(\lambda,\alpha)$ the collection of all set partitions $\AA\in \Pi_n(\lambda)$ such that $\AA$ and $\AA_\alpha$ are minimally intersecting. For every partition $\mu$, let $\mu'$ denote the corresponding conjugate partition.

  To each element in $\Pi(\mu',\alpha)$ one can associate a semistandard tableau $T\in \ssyt(\mu,\alpha)$ in a canonical way. An example will serve to illustrate the general case. Suppose $\mu=(5,3,1)$ and $\alpha=(3,3,3)$. The set partition $\AA=16|2|348|57|9$ of shape $\mu'$ is represented by the following array: 
\begin{center}
\ytableausetup{smalltableaux}
\begin{ytableau}
       1 & 2 & 3 & 5& 9 \\
       6  & \none & 4 & 7 &\none\\
       \none &\none&8&\none &\none
\end{ytableau}.
\end{center}

  Now successively perform the following steps to obtain tableaux as shown in Figure \ref{fig:partitiontossyt}.
\begin{enumerate}
\item[Step 1:] Left justify the cells in each row and then sort each row in ascending order, to obtain a tableau of shape $\mu$.
\item[Step 2:] For each positive integer $x\leq |\alpha|$, let $\SS_\alpha(x)$ denote the least integer $i$ for which $x\leq \alpha_1+\cdots+\alpha_i$. Apply $\SS_\alpha$ to each entry of the tableau obtained in Step 1 to obtain another tableau of shape $\mu$.
\end{enumerate}
  \begin{figure}[h]
    \centering
    \begin{ytableau}
       1 & 2 & 3 & 5& 9 \\
       6  & \none & 4 & 7 &\none\\
       \none &\none&8&\none &\none
     \end{ytableau}$\xrightarrow{\text{ Step 1 }}$
     \begin{ytableau}
       1 & 2 & 3 & 5& 9 \\
       4 & 6 & 7 & \none &\none\\
       8 &\none&\none&\none &\none
     \end{ytableau}$\xrightarrow{\text{ Step 2 }}$
  \begin{ytableau}
       1 & 1 & 1 & 2& 3 \\
       2 & 2 & 3 & \none &\none\\
       3 &\none&\none&\none &\none
\end{ytableau}.
    \caption{Standard and semistandard tableaux associated to $\AA=16|2|348|57|9$.}  
    \label{fig:partitiontossyt}
  \end{figure}

Somewhat remarkably, Step 1 applied to any set partition $\AA\in \Pi_n(\mu')$ always yields a standard tableau of shape $\mu$  \cite[Lem.\ 2.6]{pr}. Let $\syt(\mu,\alpha)$ denote the set of standard tableaux of shape $\mu$ such that no column of $T$ contains more than one element in any block of $\AA_\alpha$. We will prove that steps 1 and 2 above correspond to surjective maps $\TT$ and $\SS_\alpha$ as follows: 
\begin{align*}
  \Pi(\mu',\alpha) \xrightarrow{\ \displaystyle\TT\ }\syt(\mu,\alpha)\xrightarrow{\ \displaystyle\SS_\alpha\ } \ssyt(\mu,\alpha). 
\end{align*}
The maps $\TT$ and $\SS_\alpha$ for $\mu=(3,2)$ and $\alpha=(2,2,1)$ are shown in Figure \ref{fig:correspondence}.
\begin{figure}[!ht] 
  \centering
  \includegraphics[scale=.6]{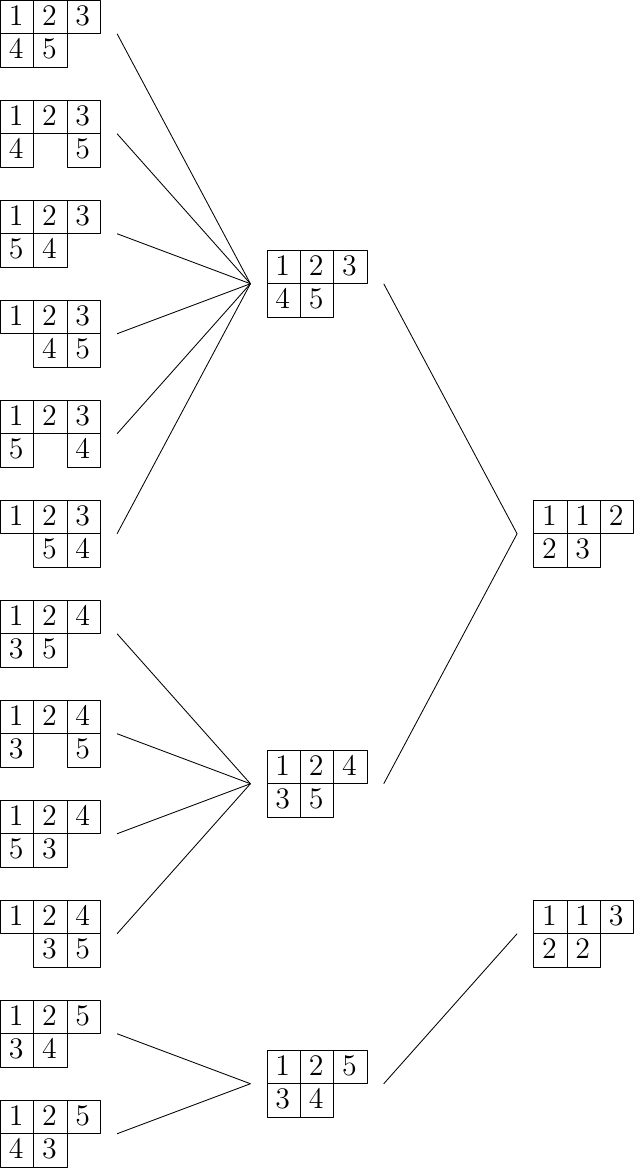}
  \caption{The set partitions, tableaux and semistandard tableaux corresponding to the maps $\TT$ and $\SS_\alpha$ for $\mu=(3,2)$ and $\alpha=(2,2,1)$.}       
  \label{fig:correspondence}
\end{figure}

\begin{proposition}
  $\TT$ maps $\Pi(\mu',\alpha)$ onto $\syt(\mu,\alpha)$.
\end{proposition}
\begin{proof}
  For each $\AA\in \Pi(\mu',\alpha)$, it follows by \cite[Lem.\ 2.6]{pr} that $\TT(\AA)$ is a standard tableau of shape $\mu$. Thus $\TT$ maps $\Pi(\mu',\alpha)$ into $\syt(\mu)$. Suppose $\alpha=(\alpha_1,\ldots,\alpha_k)$ and write $\mathcal{A}_\alpha=\cup_{i\geq 1}A_i$ where $$A_i=\{\alpha_1+\cdots+\alpha_{i-1}+1,\ldots,\alpha_1+\cdots+\alpha_{i}\},$$ 
for $1\leq i\leq k$. Suppose, by way of contradiction, that there exists a partition $\AA\in \Pi(\mu',\alpha)$ such that $\TT(\AA)$ contains a column in which two entries $a<b$ lie in the same block $A_r$ of $\AA_\alpha$. Since $A_r$ is a set of consecutive integers, all entries lying between $a$ and $b$ in the same column also lie in $A_r$. In particular, the entry $b'$ in the cell immediately above $b$ also lies in $A_r$. Suppose that the entry $b$ lies in the $i$th row and the $j$th column of $\TT(\AA)$. By the definition of $\TT$, there is a unique injective map $\theta$ from the $i$th row to the $(i-1)$th row of $\TT(\AA)$ defined by $\theta(x)=y$ if $y<x$ are consecutive elements in some block of $\AA$. Since $\AA\in \Pi(\mu',\alpha)$, it follows that $\theta(x)$ and $x$ lie in different blocks of $\AA_\alpha$. Therefore, if $x$ is in the $i$th row of $\TT(\AA)$ and $x\in A_s$ for some $s$, then the fact that $\theta(x)<x$ implies that $\theta(x)\in \cup_{i=1}^{s-1}A_i$. Since $b\in A_r$ and $b$ is in the $j$th column, it follows that the first $j$ elements in row $i$ are mapped under $\theta$ into $\cup_{i=1}^{r-1}A_i$. As $b'\in A_r$ as well, it follows that $\theta$ maps the first $j$ elements in row $i$ to elements in the $(i-1)$th row that are strictly to the left of $b'$, contradicting the injectivity of $\theta$. This proves that $\TT$ maps $\Pi(\mu',\alpha)$ into $\syt(\mu,\alpha)$.

  Surjectivity follows easily since, for every tableau $T\in \syt(\mu,\alpha),$ the set partition whose blocks are specified by the columns of $T$ maps to $T$ under $\TT$.
\end{proof}

\begin{proposition}
  $\SS_\alpha$ maps $\syt(\mu,\alpha)$ surjectively onto $\ssyt(\mu,\alpha)$.
\end{proposition}
\begin{proof}
  Suppose $\alpha=(\alpha_1,\ldots,\alpha_k)$ and let $T\in \syt(\mu,\alpha)$. Then $\SS_\alpha(T)$ is obtained from $T$ by replacing each entry $x$ by $\SS_\alpha(x)$. Since $\SS_\alpha$ is weakly order preserving, it follows that the rows and columns of $\SS_\alpha(T)$ are weakly increasing. Since no two entries in the same column of $T$ lie in the same block of $\AA_\alpha$, it follows that the entries in each column of $\SS_\alpha(T)$ are distinct and, consequently, the columns of $\SS_\alpha(T)$ are strictly increasing. It follows that $S_\alpha(T)\in \ssyt(\mu,\alpha)$. To prove the surjectivity of $\SS_\alpha$, consider a tableau $\hat{T}\in \ssyt(\mu,\alpha)$. For each integer $1\leq i\leq k$, there are $\alpha_i$ entries equal to $i$ in $\hat{T}$, each of which appears in a different column. Scanning from left to right, replace these entries by the elements in the $i$th block of $\AA_\alpha$ in increasing order. It is easily seen that this process yields a tableau $T\in \syt(\mu,\alpha)$ such that $S_\alpha(T)=\hat{T}$.
\end{proof}
We wish to describe a method of obtaining set partition statistics from certain multiset permutation statistics by using the correspondence between set partitions and semistandard tableaux. We begin by introducing some notation and defining \emph{Mahonian} statistics. The $q$-analog of the integer $n$ is defined by $[n]_q:=1+q+\cdots+q^{n-1}$. Write $[n]_q!$ for the product $\prod_{i=1}^n[i]_q$. The $q$-multinomial coefficient ${n \brack \beta_1,\ldots,\beta_r}_q$ is defined by 
  \begin{align*}
    {n \brack \beta_1,\ldots,\beta_r}_q:=\frac{[n]_q!}{[\beta_1]_q![\beta_2]_q!\cdots [\beta_r]_q!},
  \end{align*}
which is a polynomial in $q$ with nonnegative integer coefficients. Given a composition $\beta=(\beta_1,\ldots,\beta_r)$ of $n$, let $R(\beta)$ denote the class of all words of length $n$ which are rearrangements of the word $1^{\beta_1}2^{\beta_2}\cdots r^{\beta_r}$ which contains $\beta_i$ copies of $i$ for $1\leq i\leq r$. One can view $R(\beta)$ as the set of all permutations of the multiset which contains $\beta_i$ copies of $i$ $(1\leq i\leq r)$. For our purposes it will be more convenient to think of the elements of $R(\beta)$ as words. A \emph{statistic} on a set is a nonnegative integer valued function defined on the set. 
\begin{definition}\label{def:mahonian}
A statistic $\varphi$ defined on $R(\beta)$ is said to be \emph{Mahonian} \cite[Eq.\ 6.2]{foata2011q} if 
\begin{align*}
 {n \brack \beta}_q:= {n \brack \beta_1,\ldots,\beta_r}_q=\sum_{w\in R(\beta)}q^{\varphi(w)}.
\end{align*}
  \end{definition}
  Some authors refer to the property in Definition \ref{def:mahonian} as \emph{multiset Mahonian}. Examples of Mahonian statistics include the inversion number ($\inv$), the number of noninversions (${\rm ninv}$), the Major index, the $z$-index of Bressoud and Zeilberger \cite[p.\ 204]{MR791661} and Denert's index \cite{MR1061147,han1994}. We refer to the lecture notes of Foata and Han \cite{foata2011q} for more on Mahonian statistics.

  \begin{definition}
An \emph{inversion} of a word $w=w_1w_2\cdots w_n$ is defined as a pair $1\leq i<j\leq n$ such that $w_i>w_j$. The number of inversions of $w$ is denoted $\inv(w)$.
  \end{definition}

  We will show that, to each Mahonian statistic $\varphi$, one can associate a statistic on set partitions which generates the polynomials $b_{\mu\alpha}(q)$ in the introduction.
  \begin{definition}
  Let $\alpha=(\alpha_1,\ldots,\alpha_k)$ be a composition of $n$. Given a Mahonian statistic $\varphi$ and a word $w$ in its domain, define 
  \begin{align*}
    \varphi_\alpha(w):=\sum_{i=1}^k \varphi(w_i),
  \end{align*}
  where the words $w_i$ are defined uniquely by requiring that $w$ is a concatenation $w=w_1w_2\cdots w_k$ and $w_i$ has length $\alpha_i$ for $1\leq i\leq k$. 
  To each set partition $\AA\in \Pi_n$, associate a word $w=w(\AA):=a_1\cdots a_n$ where $a_i=j$ if $i$ is the $j$-th smallest element in some block of $\AA$. Define
\begin{align*}
  \varphi_{\alpha}(\AA):=\varphi_\alpha(w(\AA)).
\end{align*}
  \end{definition}
\begin{example}
If $\AA=127|34|56$, then $w(\AA)=1212123$. Suppose $\varphi=\inv$ and $\alpha=(3,2,1,1)$. Then $\varphi_\alpha(\AA)=\inv(121)+\inv(21)+\inv(2)+\inv(3)=1+1+0+0=2.$  
\end{example}

Consider the set partitions 
\begin{align*}
\AA&=\begin{ytableau}
       1 & 2 & 3 & 5& 9 \\
       6  & \none & 4 & 7 &\none\\
       \none &\none&8&\none &\none
\end{ytableau}, \qquad
\BB=\begin{ytableau}
       1 & 2 & 3 & 5& 9 \\
       4  & \none & 6 & 7 &\none\\
       8 &\none&\none&\none &\none
\end{ytableau}.
\end{align*}
Note that $w(\AA)=111212231 =w(\BB)$ since each entry appears in the same position within blocks of $\AA$ and $\BB$. More generally, we have the following result.
\begin{proposition}\label{prop:onlysyt}
For $\AA,\BB\in \Pi_n$, $w(\AA)=w(\BB)$ if and only if $\TT(\AA)=\TT(\BB)$. In particular, if $\TT(\AA)=\TT(\BB)$, then for every composition $\alpha$ and every Mahonian statistic $\varphi$, we have $\varphi_\alpha(\AA)=\varphi_\alpha(\BB)$. 
\end{proposition}
\begin{proof}
  An element is the $i$th smallest element in some block of $\AA$ if and only if it appears in the $i$th row of $\TT(\AA)$. Therefore, each element appears in the same position in blocks of $\AA$ and $\BB$ if and only if $\TT(\AA)=\TT(\BB)$. 
\end{proof}

In view of Proposition \ref{prop:onlysyt}, we can define $w(T)$ for any standard tableau $T$.
\begin{definition}\label{def:tabtoword}
To each standard tableau $T$, associate a word $w(T)=a_1a_2\cdots a_n$ where $a_i=j$ if the entry $i$ occurs in the $j$th row of $T$.  For each Mahonian statistic $\varphi$, define $$\varphi_\alpha(T)=\varphi_\alpha(w(T)).$$
\end{definition}
 It should be clear from the definitions above that, for each set partition $\AA$, we have $w(\AA)=w(\TT(\AA))$ and $\varphi_\alpha(\AA)=\varphi_\alpha(\TT(\AA))$. 
\begin{example}
  For the standard tableau
  \begin{align*}
T&=\begin{ytableau}
       1 & 2 & 3 & 5& 9 \\
       4  & 6 & 7 & \none &\none\\
       8 &\none&\none&\none &\none
\end{ytableau},
\end{align*}
we have $w(T)=111212231.$ If $\alpha=(3,3,3)$ and $\varphi=\inv$, then $\varphi_{\alpha}(T)=\inv(111)+\inv(212)+\inv(231)=0+1+2=3.$ 
\end{example}

\begin{definition}\label{def:bit}
Suppose $T$ is a semistandard tableau. For each $i\geq 1$, let $\beta^i_{j}=\beta^i_j(T)$ denote the number of occurrences of the integer $i$ in the $j$th row of $T$ for $j\geq 1$. Define $\beta^i=\beta^i(T):=(\beta^i_{1},\beta^i_{2},\ldots)$ for $i\geq 1$. 
\end{definition}

\begin{example}
  Let $\mu=(5,3,1)$ and $\alpha=(3,3,3)$. Consider the tableau $T\in \ssyt(\mu,\alpha)$ given by
  \begin{align*}
T=\begin{ytableau}
       1 & 1 & 1 & 2& 3 \\
       2 & 2 & 3 & \none &\none\\
       3 &\none&\none&\none &\none
\end{ytableau}.
  \end{align*}
  In this case $\beta^1=(3),\beta^2=(1,2)$ and $\beta^3=(1,1,1).$ Consider the map $\SS_\alpha:\syt(\mu,\alpha)\to \ssyt(\mu,\alpha)$. The fiber of $\SS_\alpha$ over $T$ contains the tableau
  \begin{align*}
\hat{T}=\begin{ytableau}
       1 & 2 & 3 & 5 & 8 \\
       4 & 6 & 7 & \none &\none\\
       9 &\none&\none&\none &\none
     \end{ytableau}.
  \end{align*}
  In this case $w(\hat{T})$ is a concatenation $w_1w_2w_3$ where $w_1,w_2,w_3$ are given by $111, 212, 213 $ respectively. Note that the length of $w_i$ is $\alpha_i$ and $w_i\in R(\beta^i)$ for $1\leq i\leq 3$.
\end{example}
  In general, given compositions $\beta^1,\ldots,\beta^k$ and a Mahonian statistic $\varphi$, assign a weight to tuples $(w_1,\ldots,w_k)$ in the Cartesian product  $R(\beta^1)\times \cdots\times R(\beta^k)$ by $\varphi((w_1,\ldots,w_k))=\sum_{i=1}^k\varphi(w_i).$ Given a word $w$ of length $n$, define $g_\alpha(w)=(w_1,\ldots,w_k)$ where the $w_i$ are determined uniquely by requiring that $w_i$ has length $\alpha_i$ and $w=w_1w_2\cdots w_k$. The following proposition characterizes the fibers $\SS_\alpha^{-1}(T)$.

\begin{proposition}\label{prop:bijection}
Let $T\in \ssyt(\mu,\alpha)$ and suppose $\alpha$ has length $k$. Let $\beta^i=\beta^i(T)$ $(1\leq i\leq k)$ be as above. For each Mahonian statistic $\varphi$, the map $\hat{T}\mapsto g_\alpha(w(\hat{T}))$ is a weight preserving bijection between the fiber $\SS_\alpha^{-1}(T)$ weighted by $\varphi_\alpha$ and the Cartesian product $R(\beta^1)\times \cdots \times R(\beta^k)$ weighted by $\varphi$. 
\end{proposition}
\begin{proof}
  It follows from the definition of $\varphi_\alpha$ that the map $\hat{T}\mapsto g_\alpha(w(\hat{T}))$ is weight preserving. Moreover, it easily seen that if $\hat{T}\in \SS_\alpha^{-1}(T)$, then $g_\alpha(w(\hat{T}))\in R(\beta^1)\times \cdots \times R(\beta^k).$ It remains to prove that the map $\hat{T}\mapsto g_\alpha(w(\hat{T}))$ is a bijection from $\SS_\alpha^{-1}(T)$ to $\prod_{i=1}^kR(\beta^i).$
First note that this map is injective since the map $\hat{T}\mapsto w(\hat{T})$ is injective. Since the cardinality of $\prod_{i=1}^kR(\beta^i)$ is given by the product
  \begin{align*}
  N= \prod_{i=1}^k{\alpha_i \choose \beta^i}
  \end{align*}
of multinomial coefficients, surjectivity would follow if we prove that there are $N$ distinct elements in the fiber $\SS_\alpha^{-1}(T)$. Consider the canonical partition $\AA_\alpha=\cup_{i=1}^kA_i$ where $|A_i|=\alpha_i$ $(1\leq i\leq k)$. Each element in the fiber $\SS_\alpha^{-1}(T)$ can be obtained from $T$ by replacing all entries equal to $i$ in $T$ by a suitable permutation of elements in $A_i$ for each $1\leq i\leq k$. Successively for $i=1,2,\ldots,k$, write $A_i=B_{i1}\cup B_{i2}\cup\cdots\cup B_{ii}$ as an ordered union of $i$ disjoint sets (possibly empty) with $|B_{ij}|=\beta^i_{j}$ and replace the $\beta^i_{j}$ entries equal to $i$ in row $j$ of $T$ by the entries of $B_{ij}$ in ascending order. It is easily seen that this procedure always results in a standard tableau ${\tilde T}\in \SS_\alpha^{-1}(T).$ Since the number of ways to write the $A_i$ as an ordered union of disjoint sets above is the multinomial coefficient ${\alpha_i \choose \beta^i}$ for each $1\leq i\leq k$, it follows that there exist $N$ standard tableaux in $\SS_\alpha^{-1}(T)$, proving surjectivity.
\end{proof}

\begin{example}
    Let $\mu=(5,3,1)$ and $\alpha=(3,3,3)$. Consider the tableau $T\in \ssyt(\mu,\alpha)$ defined by
  \begin{align*}
T=\begin{ytableau}
       1 & 1 & 1 & 2& 3 \\
       2 & 2 & 3 & \none &\none\\
       3 &\none&\none&\none &\none
\end{ytableau}. 
  \end{align*}
To enumerate tableaux in $\SS_\alpha^{-1}(T)$ we compute $\beta^1=(3),\beta^2=(1,2)$ and $\beta^3=(1,1,1).$ Here $\AA_\alpha=123|456|789=A_1\cup A_2\cup A_3$. As $\beta^1=(3)$, there is a unique way to write $A_1$ as a union of a single set of size 3. This corresponds to replacing the entries equal to 1 in $T$ by the elements in $A_1$:
  \begin{align*}
    \begin{ytableau}
       1 & 2 & 3 &  &  \\
        &  &  & \none & \none \\
        &\none & \none & \none & \none
\end{ytableau}.
  \end{align*}
  Since $\beta^2=(1,2)$, we may write $A_2=\{4,5,6\}$ as a union in 3 ways, namely $\{4\}\cup\{5,6\}$ and $\{5\}\cup \{4,6\}$ and $\{6\}\cup \{4,5\}.$ These correspond respectively to the following replacements of entries equal to 2 in $T$: 
  \begin{align*}
    \begin{ytableau}
       1 & 2 & 3 &4  &  \\
       5 & 6 &  & \none & \none \\
        &\none & \none & \none & \none
      \end{ytableau},\quad
    \begin{ytableau}
       1 & 2 & 3 &5  &  \\
       4 & 6 &  & \none & \none \\
        &\none & \none & \none & \none
\end{ytableau} \mbox{ and }
    \begin{ytableau}
       1 & 2 & 3 & 6 &  \\
        4&5   &  & \none & \none \\
        &\none & \none & \none & \none
\end{ytableau}.
  \end{align*}
In each of the above tableaux, the empty cells can be filled in ${\alpha_3\choose \beta^3}={3 \choose 1,1,1}=6$ ways to obtain an element in $\SS_\alpha^{-1}(T)$. It follows that $|\SS_\alpha^{-1}(T)|=3\cdot 6=18,$ which is equal to the product $\prod_{i=1}^3{\alpha_i \choose \beta^i}.$
\end{example}
We now define a $q$-deformation of the product $\prod_{i=1}^k{\alpha_i \choose \beta^i}$ which will be used to define the polynomials $b_{\mu\alpha}(q)$ in Section \ref{sec:setstat}. 
\begin{definition}\label{def:sqt}
Given $T\in \ssyt(\mu,\alpha)$, let $\beta^i=\beta^i(T)$ and define
\begin{align*}
  s_q(T):=\prod_{i\geq 1}{\alpha_i \brack \beta^i}_q.
\end{align*}
\end{definition}
\begin{lemma}\label{lem:sqt}
  For each Mahonian statistic $\varphi$ and each tableau $T\in \ssyt(\mu,\alpha)$, we have
  \begin{align*}
  \sum_{\substack{\hat{T}\in \syt(\mu,\alpha)\\ \SS_\alpha(\hat{T})=T}}q^{\varphi_\alpha(\hat{T})}=s_q(T).
  \end{align*}
\end{lemma}
\begin{proof}
Suppose $\alpha=(\alpha_1,\ldots,\alpha_k)$. By Proposition \ref{prop:bijection},  we have
  \begin{align*}
    \sum_{\substack{\hat{T}\in \syt(\mu,\alpha)\\ \SS_\alpha(\hat{T})=T}}q^{\varphi_\alpha(\hat{T})}&=\sum_{(w_1,\ldots,w_k)\in \prod_{i\geq 1}R(\beta^i)} q^{\varphi(w_1)+\cdots+\varphi(w_k)}\\
                                                                                    &=\prod_{i\geq 1}\sum_{w_i \in R(\beta^i)}q^{\varphi(w_i)}\\
                                                                                    &=\prod_{i\geq 1}{\alpha_i \brack \beta^i}_q\\
    &=s_q(T). \qedhere
  \end{align*}
\end{proof}   
\section{Set partitions statistics arising from Mahonian statistics}\label{sec:setstat}  
In this section we show that to each Mahonian statistic on multiset permutations, one can associate a statistic on set partitions. We require the notion of \emph{interlacing number} of a set partition from \cite[Sec.\ 3.3]{pr} which we now describe. Given a set partition $\AA\in \Pi_n$, we form its extended chord diagram as follows. First plot the points $1,\ldots,n$ on a number line with an additional point to the right, denoted $\infty$, which is deemed to be greater than all positive integers. To each block $B=\{b_1,\ldots,b_r\}$ of $\AA$ of cardinality $r$, we associate $r$ arcs, $(b_1,b_2), (b_2,b_3),\ldots,(b_r,\infty).$ For $1\leq j\leq r$, the arc $(b_j,b_{j+1})$ is referred to as the $j$-th arc of $B$ (with the convention $b_{r+1}=\infty$). For instance, when $\AA=127|36|45,$ its extended chord diagram is shown in Figure \ref{fig:extdiag}. A pair of arcs $\{(a_1,a_2),(b_1,b_2)\}$ is said to be \emph{crossing} if either $a_1<b_1<a_2<b_2$ or $b_1<a_1<b_2<a_2$ (Kreweras \cite{MR0309747}). An \emph{interlacing} of $\AA$ is a crossing of $j$-th arcs of distinct blocks of $\AA$ for some $j$. For instance, the extended chord diagram in Figure~\ref{fig:extdiag} has precisely 2 interlacings corresponding to the arc pairs $\{(2,7),(5,\infty)\}$ and $\{(2,7),(6,\infty)\}$.  
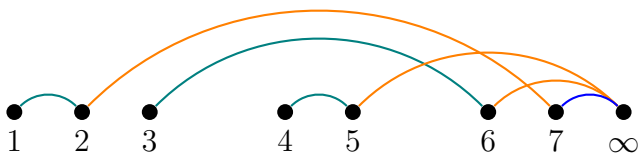
\begin{figure}[ht]
\begin{tikzpicture}
                        [scale=0.9,every node/.style={circle,fill=black,,inner sep=2pt, minimum size=6pt}]
                        \node[label=below:$1$] (1) at (1,0) {};
                        \node[label=below:$2$] (2) at (2,0) {};
                        \node[label=below:$3$] (3) at (3,0) {};
                        \node[label=below:$4$] (5) at (5,0) {};
                        \node[label=below:$5$] (6) at (6,0) {};
                        \node[label=below:$6$] (8) at (8,0) {};
                        \node[label=below:$7$] (9) at (9,0) {};
                        \node[label=below:$\infty$] (0) at (10,0) {};
                        \draw[thick,color=teal]
                        (1) [out=45, in=135] to  (2);
                        \draw[thick,color=orange]
                        (2) [out=45, in=135] to  (9);
                        \draw[thick,color=teal]
                        (3) [out=45, in=135] to  (8);
                        \draw[thick,color=teal]
                        (5) [out=45, in=135] to  (6);
                        \draw[thick,color=orange]
                        (6) [out=45, in=135] to  (0);
                        \draw[thick,color=orange]
                        (8) [out=45, in=135] to  (0);
                        \draw[thick,color=blue]
                        (9) [out=45, in=135] to  (0);
\end{tikzpicture}
                      \caption{Extended chord diagram of the partition $127|36|45$. The $j$th arcs in each block are colored identically for $j\geq 1$.}
                      \label{fig:extdiag}
\end{figure}  
\begin{definition}
Let $\alpha$ be a composition of $n$ and $\AA\in \Pi_n$. Two interlacing arcs $(a,b)$ and $(c,d)$ of $\AA$ with $a<c<b<d$ are said to be $\alpha$-interlacing if $b$ and $c$ do not belong to the same block of $\mathcal{A}_\alpha$.
\end{definition}
 For instance, the arcs $(2,7)$ and $(6,\infty)$ in the extended chord diagram of Figure \ref{fig:extdiag} are $\alpha$-interlacing when $\alpha=(3,2,1,1)$, but not when $\alpha=(2,3,2)$. Let $i_\alpha(\AA)$ denote the number of $\alpha$-interlacings of $\AA$. Values of $i_\alpha(\AA)$ for some set partitions are shown in Table \ref{tab:statistics}.
\begin{remark}\label{rem:crossings}
  For $\alpha=(1^n),$ $\alpha$-interlacings are simply interlacings. In addition, if all blocks of $\AA$ are of cardinality 2, then interlacings of $\AA$  are precisely the crossings of $\AA$. 
\end{remark}
 We will show that the sum $\sum q^{i_\alpha(\AA)}$ taken over all set partitions $\AA\in \Pi(\mu',\alpha)$ which map to a fixed tableau under $\TT$ can be expressed compactly as a product of suitably defined $q$-integers. An example will serve to illustrate the general case.  

\begin{example}
  Suppose $\mu=(6,2)$ and $\alpha=(4,4)$ and consider the tableau $T\in \syt(\mu,\alpha)$ given by
  \begin{align*}
    T=
    \begin{ytableau}
      1&2&3&4&5&7\\
      6&8&\none
    \end{ytableau}.
  \end{align*}
  The canonical set partition corresponding to $\alpha$ is $\AA_\alpha=1234|5678=A_1\cup A_2.$ We wish to compute the sum
  \begin{align*}
 \sum_{\substack{\AA\in \Pi(\mu',\alpha)\\ \TT(\AA)=T}}q^{i_\alpha(\AA)}.
  \end{align*}
  Let $R_1,R_2$ denote the set of entries in the first and second rows of $T$ respectively. Any set partition $\AA\in \Pi(\mu',\alpha)$ with $\TT(\AA)=T$ is uniquely specified by a map $\theta:R_2\to R_1$ with the property that if $x\in A_j$, then $\theta(x)\in A_i$ for some $i<j$. Thus $\theta(6),\theta(8)\in \{1,2,3,4\}$. The choices $\theta(6)=4,3,2,1$ contribute $0,1,2,3$ $\alpha$-interlacings to $i_\alpha(\AA)$ respectively (the arc passing through 5 does not contribute to interlacings since 5 and 6 lie in the same block of $\AA_\alpha$). For instance, if $\theta(6)=2$, then there are two $\alpha$-interlacings formed by arcs which join $3$ and $4$ to either 8 or $\infty$. Once $\theta(6)$ has been chosen, there are 3 remaining choices for $\theta(8)$. For instance, if $\theta(6)=3$, then $\theta(8) \in \{1,2 , 4\}$ as shown in Figure \ref{fig:choices}. The choices $\theta(8)=4,2,1$ result in a contribution of $0,1,2$ to $i_\alpha(\AA)$ respectively. The extended chord diagram for $\theta(6)=3,\theta(8)=1$ is shown in Figure \ref{fig:choices1}.
    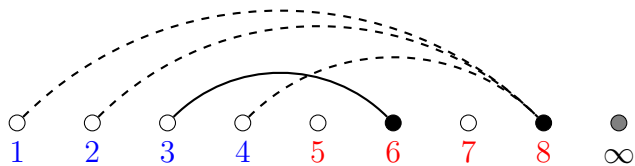
\begin{figure}[h]
    \begin{displaymath}
      \begin{tikzpicture}
        [scale=1,every node/.style={circle,draw,,inner sep=2pt, minimum size=6pt}]
        \node[label=below:$\textcolor{blue}{1}$] (1) at (1,0) {}; 
        \node[label=below:$\textcolor{blue}{2}$] (2) at (2,0) {};
        \node[label=below:$\textcolor{blue}{3}$] (3) at (3,0) {};
        \node[label=below:$\textcolor{blue}{4}$] (4) at (4,0) {};
        \node[label=below:$\textcolor{red}{5}$] (5) at (5,0) {};
        \node[label=below:$\textcolor{red}{6}$, fill=black] (6) at (6,0) {};
        \node[label=below:$\textcolor{red}{7}$] (7) at (7,0) {};
        \node[label=below:$\textcolor{red}{8}$, fill=black] (8) at (8,0) {};
        \node[label=below:$\infty$, fill=gray] (0) at (9,0) {};
        \draw[thick,color=black]
        (3) [out=45, in=135] to  (6);
        \draw[thick,color=black, style=dashed]
        (1) [out=45, in=135] to  (8);
              \draw[thick,color=black, style=dashed]
              (2) [out=45, in=135] to  (8);
              \draw[thick,color=black, style=dashed]
              (4) [out=45, in=135] to  (8);
      \end{tikzpicture}
    \end{displaymath}
    \caption{Choices for $\theta(8)$ when $\theta(6)=3$. Numbers of the same color are in the same block of $\AA_\alpha$. White and black nodes correspond to elements in the first and second row of the tableau respectively.}   
    \label{fig:choices}
  \end{figure}
 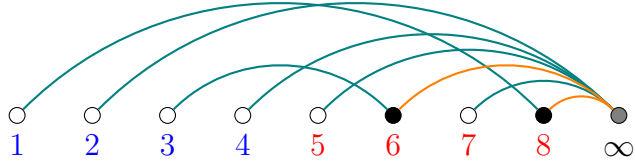
\begin{figure}[h]
    \begin{displaymath}
      \begin{tikzpicture}
        [scale=1,every node/.style={circle,draw,,inner sep=2pt, minimum size=6pt}]
        \node[label=below:$\textcolor{blue}{1}$] (1) at (1,0) {};
        \node[label=below:$\textcolor{blue}{2}$] (2) at (2,0) {};
        \node[label=below:$\textcolor{blue}{3}$] (3) at (3,0) {};
        \node[label=below:$\textcolor{blue}{4}$] (4) at (4,0) {};
        \node[label=below:$\textcolor{red}{5}$] (5) at (5,0) {};
        \node[label=below:$\textcolor{red}{6}$, fill=black] (6) at (6,0) {};
        \node[label=below:$\textcolor{red}{7}$] (7) at (7,0) {};
        \node[label=below:$\textcolor{red}{8}$, fill=black] (8) at (8,0) {};
        \node[label=below:$\infty$, fill=gray] (0) at (9,0) {};
        \draw[thick,color=teal]
        (3) [out=45, in=135] to  (6);
        \draw[thick,color=teal]
        (1) [out=45, in=135] to  (8);
        \draw[thick,color=teal]
        (2) [out=45, in=135] to  (0);
        \draw[thick,color=teal]
        (4) [out=45, in=135] to  (0);
        \draw[thick,color=teal]
        (5) [out=45, in=135] to  (0);
        \draw[thick,color=teal]
        (7) [out=45, in=135] to  (0);
        \draw[thick,color=orange]
        (6) [out=45, in=135] to  (0);
        \draw[thick,color=orange]
        (8) [out=45, in=135] to  (0);
      \end{tikzpicture}
    \end{displaymath}
    \caption{Extended chord diagram for $\AA=18|2|36|4|5|7$. First arcs and second arcs are colored differently. When $\alpha=(4,4)$, $\alpha$-interlacings are given by $\{(3,6),(4,\infty)\}$, $\{(1,8),(2,\infty)\}$ and $\{(1,8),(4,\infty)\}.$ Thus $i_\alpha(\AA)=3$.}    
    \label{fig:choices1}
  \end{figure}

  Therefore,
    \begin{align*}
 \sum_{\substack{\AA\in \Pi(\mu',\alpha)\\ \TT(\AA)=T}}q^{i_\alpha(\AA)}=(1+q+q^2+q^3)(1+q+q^2)=[4]_q[3]_q.
  \end{align*}
\end{example}
\begin{definition}\label{def:rqt}
Let $b_{ij}$ denote the entry in row $i$ and column $j$ of a tableau $T$. For each $i\geq 2$, the cell $(i,j)$ is assigned a value given by
\begin{align*}
  r_{ij}(T):=\#\{j':j'\geq j\mbox{ and }b_{i-1,j'}<b_{ij}\}.
\end{align*}
Define
\begin{align*}
  r_q(T):=\prod_{i\geq 2}\prod_{j\geq 1}[r_{ij}(T)]_q.
\end{align*}  
\end{definition}

\begin{example}\label{eg:atab}
  For
  \begin{align*}
   T=\ytableaushort{111223,223\none,34 \none},    
  \end{align*}
  the values of $r_{ij}(T)$ for each corresponding cell are indicated below:
\begin{align*}
   \ytableaushort{\ \ \ \ \ \ ,323\none,22 \none}.
\end{align*}
In this case $r_q(T)=(1+q)^3(1+q+q^2)^2.$ 
  \end{example} 

\begin{proposition}\label{prop:interlace}
  Suppose $\hat{T}\in \syt(\mu,\alpha)$ and $\SS_\alpha(\hat{T})=T$. Then
  \begin{align*}
   \sum_{\substack{\AA\in \Pi(\mu',\alpha)\\ \TT(\AA)=\hat{T}}}q^{i_\alpha(\AA)}=r_q(T). 
  \end{align*}
\end{proposition}
\begin{proof}
Let $F(\hat{T})$ denote the sum in the statement of the proposition. Suppose the canonical set partition corresponding to $\alpha$ is $\AA_\alpha=\cup_{i=1}^kA_i.$ Let $R_i$ denote the entries in the $i$th row of $\hat{T}$. Each set partition $\AA\in \Pi(\mu',\alpha)$ with $\TT(\AA)=\hat{T}$ is uniquely specified by maps $\theta_i:R_{i}\to R_{i-1}$ $(i\geq 2)$ with the property that if $x\in R_p$ and $x\in A_s$, then $\theta_p(x)\in A_r$ for some $r<s$. Suppose $\hat{T}=(a_{ij})$ where $a_{ij}$ denotes the $j$th entry in the $i$th row of $\hat{T}$. Then $T=(b_{ij})$ where $b_{ij}$ is the unique integer for which $a_{ij}\in A_{b_{ij}}$. The number of choices for $\theta_2(a_{21})$ is equal to  
  \begin{align*}
   \# \{j':b_{1j'}<b_{21}\}=r_{21}(T).
  \end{align*}
  Thus the contribution of these choices to the sum $F(\hat{T})$ is given by
  \begin{align*}
    1+q+\cdots+q^{r_{21}(T)-1}=[r_{21}(T)]_q.
  \end{align*}
  Once $\theta_2(a_{21})$ has been chosen, the number of choices for $\theta_2(a_{22})$ is equal to
  \begin{align*}
    &\# \{j':b_{1j'}<b_{22}\}-1\\
    &=\# \{j':j'\geq 2 \mbox{ and } b_{1j'}<b_{22}\}=r_{22}(T).
  \end{align*}
  These choices contribute $[r_{22}(T)]_q$ to $F(\hat{T})$. In general, once $\theta_2$ has been defined for $a_{21},a_{22},\ldots,a_{2,j-1}$, the number of choices for $\theta_2(a_{2j})$ is equal to
    \begin{align*}
    &\# \{j':b_{1j'}<b_{2j}\}-(j-1)\\
    &=\# \{j':j'\geq j \mbox{ and } b_{1j'}<b_{2j}\}=r_{2j}(T).
  \end{align*}
These choices contribute $[r_{2j}(T)]_q$ to $F(\hat{T})$. Once $\theta_2$ has been defined for all elements in the second row of $\hat{T}$ we proceed to define $\theta_3$ starting with the entry $a_{31}$. Continuing this line of reasoning, it can be seen that the number of choices for $\theta_i(a_{ij})$ is equal to
    \begin{align*}
    &\# \{j':b_{i-1,j'}<b_{ij}\}-(j-1)\\
    &=\# \{j':j'\geq j \mbox{ and } b_{i-1,j'}<b_{ij}\}=r_{ij}(T),
  \end{align*}
and these choices contribute $[r_{ij}(T)]_q$ to $F(\hat{T})$. Therefore
  \begin{align*}
    F(\hat{T})&=\prod_{i\geq 2}\prod_{j\geq 1}[r_{ij}(T)]_q=r_q(T). \qedhere
  \end{align*} 
\end{proof}

In the rest of this section, $t$ denotes a formal variable. We are now in a position to define the polynomials $b_{\mu\alpha}(q)$ mentioned in the introduction.
 \begin{definition}\label{def:bmualpha}
For each partition $\mu$ and composition $\alpha$ with $|\mu|=|\alpha|$, define
\begin{align*}
  \tilde{b}_{\mu\alpha}(q,t):= \sum_{T\in \ssyt(\mu,\alpha)}r_q(T)s_t(T),
\end{align*}
and
\begin{align*}
  b_{\mu\alpha}(q):= \tilde{b}_{\mu\alpha}(q,q).  
\end{align*}
\end{definition}
 Let ${\rm sort}(\alpha)$ denote the partition obtained by sorting the sequence $\alpha_i$ $(i\geq 1)$ in weakly decreasing order. Recall that the Kostka number $K_{\mu\alpha}$ is defined as the number of semistandard tableaux of shape $\mu$ and content $\alpha$. By setting $q=0$ in Definition \ref{def:bmualpha}, it follows that $b_{\mu\alpha}(0)=K_{\mu\alpha}$. It follows from Definition \ref{def:bmualpha} that $b_{\mu\alpha}(q)$ is nonzero precisely when $K_{\mu\alpha}>0$ which happens when $\mu\geq {\rm sort}(\alpha)$ in the dominance order on partitions.

\begin{example}\label{eg:diagram}
Suppose $\mu=(5,1)$ and $\alpha=(3,3)$. In this case $\ssyt(\mu,\alpha)$ consists of the single tableau $T=\ytableaushort{11122,2}$ with $r_q(T)=s_q(T)=[3]_q$. Therefore
  \begin{align*}
    b_{\mu\alpha}(q)=r_q(T)s_q(T)=q^4+2q^3+3q^2+2q+1.
  \end{align*}
 We wish to show that the polynomials $b_{\mu\alpha}(q)$ defined above can also be obtained from several set partition statistics. Let $\inv(w)$ denote the number of inversions of $w$ and suppose $\mu=(5,1)$ and $\alpha=(3,3)$ are as in Example \ref{eg:diagram} above. The values of $i_\alpha$ and $\inv_\alpha$ for set partitions in $\Pi(\mu',\alpha)$ are shown in Table \ref{tab:statistics}. In this case it is easily verified that
  \begin{align*}
b_{\mu\alpha}(q)=   \sum_{\AA\in \Pi(\mu',\alpha)}q^{i_{\alpha}(\AA)+\inv_\alpha(\AA)}.
  \end{align*}
\end{example}
In general, we have the following result.
\begin{table}
\centering  
    \begin{tabular}{|cccccc|}
      \hline
      $\AA$ & Diagram of $\AA$ & $w(\AA)$ &$i_\alpha(\AA)$ & $\varphi_\alpha(\AA)$&$i_\alpha(\AA)+\varphi_\alpha(\AA)$\\
      \hline
            $14|2|3|5|6$ &
\begin{minipage}{0.45\textwidth}
                      \begin{tikzpicture}
                        [scale=1,every node/.style={circle,fill=black,inner sep=2pt, minimum size=6pt}]
                        \node[label=below:$\textcolor{blue}{1}$] (1) at (1,0) {};
                        \node[label=below:$\textcolor{blue}{2}$] (2) at (2,0) {};
                        \node[label=below:$\textcolor{blue}{3}$] (3) at (3,0) {};
                        \node[label=below:$\textcolor{red}{4}$] (4) at (4,0) {};
                        \node[label=below:$\textcolor{red}{5}$] (5) at (5,0) {};
                        \node[label=below:$\textcolor{red}{6}$] (6) at (6,0) {};
                        \node[label=below:$\infty$] (0) at (7,0) {};
                        \draw[thick,color=teal]
                        (2) [out=45, in=135] to  (0);
                        \draw[thick,color=teal]
                        (3) [out=45, in=135] to  (0);
                        \draw[thick,color=teal]
                        (5) [out=45, in=135] to  (0);
                        \draw[thick,color=teal]
                        (6) [out=45, in=135] to  (0);
                        \draw[thick,color=teal]
                        (1) [out=45, in=135] to  (4);
                        \draw[thick,color=orange]
                        (4) [out=45, in=135] to  (0);
                      \end{tikzpicture}
                    \end{minipage}      &$111\cdot 211$                      & 2&2&4\\
      $15|2|3|4|6$ &
\begin{minipage}{0.45\textwidth}
                      \begin{tikzpicture}
                        [scale=1,every node/.style={circle,fill=black,inner sep=2pt, minimum size=6pt}]
                        \node[label=below:$\textcolor{blue}{1}$] (1) at (1,0) {};
                        \node[label=below:$\textcolor{blue}{2}$] (2) at (2,0) {};
                        \node[label=below:$\textcolor{blue}{3}$] (3) at (3,0) {};
                        \node[label=below:$\textcolor{red}{4}$] (4) at (4,0) {};
                        \node[label=below:$\textcolor{red}{5}$] (5) at (5,0) {};
                        \node[label=below:$\textcolor{red}{6}$] (6) at (6,0) {};
                        \node[label=below:$\infty$] (0) at (7,0) {};
                        \draw[thick,color=teal]
                        (2) [out=45, in=135] to  (0);
                        \draw[thick,color=teal]
                        (3) [out=45, in=135] to  (0);
                        \draw[thick,color=teal]
                        (4) [out=45, in=135] to  (0);
                        \draw[thick,color=teal]
                        (6) [out=45, in=135] to  (0);
                        \draw[thick,color=teal]
                        (1) [out=45, in=135] to  (5);
                        \draw[thick,color=orange]
                        (5) [out=45, in=135] to  (0);
                      \end{tikzpicture}
                    \end{minipage}             &$111\cdot 121$               & 2&1&3\\
            $16|2|3|4|5$ &
\begin{minipage}{0.45\textwidth}
                      \begin{tikzpicture}
                        [scale=1,every node/.style={circle,fill=black,inner sep=2pt, minimum size=6pt}]
                        \node[label=below:$\textcolor{blue}{1}$] (1) at (1,0) {};
                        \node[label=below:$\textcolor{blue}{2}$] (2) at (2,0) {};
                        \node[label=below:$\textcolor{blue}{3}$] (3) at (3,0) {};
                        \node[label=below:$\textcolor{red}{4}$] (4) at (4,0) {};
                        \node[label=below:$\textcolor{red}{5}$] (5) at (5,0) {};
                        \node[label=below:$\textcolor{red}{6}$] (6) at (6,0) {};
                        \node[label=below:$\infty$] (0) at (7,0) {};
                        \draw[thick,color=teal]
                        (2) [out=45, in=135] to  (0);
                        \draw[thick,color=teal]
                        (3) [out=45, in=135] to  (0);
                        \draw[thick,color=teal]
                        (4) [out=45, in=135] to  (0);
                        \draw[thick,color=teal]
                        (5) [out=45, in=135] to  (0);
                        \draw[thick,color=teal]
                        (1) [out=45, in=135] to  (6);
                        \draw[thick,color=orange]
                        (6) [out=45, in=135] to  (0);
                      \end{tikzpicture}
                    \end{minipage}                   &$111\cdot 112$         & 2&0&2\\
      $1|24|3|5|6$ & 
                    \begin{minipage}{0.45\textwidth}
                      \begin{tikzpicture}
                        [scale=1,every node/.style={circle,fill=black,inner sep=2pt, minimum size=6pt}]
                        \node[label=below:$\textcolor{blue}{1}$] (1) at (1,0) {};
                        \node[label=below:$\textcolor{blue}{2}$] (2) at (2,0) {};
                        \node[label=below:$\textcolor{blue}{3}$] (3) at (3,0) {};
                        \node[label=below:$\textcolor{red}{4}$] (4) at (4,0) {};
                        \node[label=below:$\textcolor{red}{5}$] (5) at (5,0) {};
                        \node[label=below:$\textcolor{red}{6}$] (6) at (6,0) {};
                        \node[label=below:$\infty$] (0) at (7,0) {};
                        \draw[thick,color=teal]
                        (1) [out=45, in=135] to  (0);
                        \draw[thick,color=teal]
                        (3) [out=45, in=135] to  (0);
                        \draw[thick,color=teal]
                        (5) [out=45, in=135] to  (0);
                        \draw[thick,color=teal]
                        (6) [out=45, in=135] to  (0);
                        \draw[thick,color=teal]
                        (2) [out=45, in=135] to  (4);
                        \draw[thick,color=orange]
                        (4) [out=45, in=135] to  (0);
                      \end{tikzpicture}
                    \end{minipage}&$111\cdot 211$                            &  1&2&3\\
     $1|25|3|4|6$&
\begin{minipage}{0.45\textwidth}
                      \begin{tikzpicture}
                        [scale=1,every node/.style={circle,fill=black,inner sep=2pt, minimum size=6pt}]
                        \node[label=below:$\textcolor{blue}{1}$] (1) at (1,0) {};
                        \node[label=below:$\textcolor{blue}{2}$] (2) at (2,0) {};
                        \node[label=below:$\textcolor{blue}{3}$] (3) at (3,0) {};
                        \node[label=below:$\textcolor{red}{4}$] (4) at (4,0) {};
                        \node[label=below:$\textcolor{red}{5}$] (5) at (5,0) {};
                        \node[label=below:$\textcolor{red}{6}$] (6) at (6,0) {};
                        \node[label=below:$\infty$] (0) at (7,0) {};
                        \draw[thick,color=teal]
                        (1) [out=45, in=135] to  (0);
                        \draw[thick,color=teal]
                        (3) [out=45, in=135] to  (0);
                        \draw[thick,color=teal]
                        (4) [out=45, in=135] to  (0);
                        \draw[thick,color=teal]
                        (6) [out=45, in=135] to  (0);
                        \draw[thick,color=teal]
                        (2) [out=45, in=135] to  (5);
                        \draw[thick,color=orange]
                        (5) [out=45, in=135] to  (0);
                      \end{tikzpicture}
                    \end{minipage}       &$111\cdot 121$                     & 1&1&2\\
      $1|26|3|4|5$ &
\begin{minipage}{0.45\textwidth}
                      \begin{tikzpicture}
                        [scale=1,every node/.style={circle,fill=black,inner sep=2pt, minimum size=6pt}]
                        \node[label=below:$\textcolor{blue}{1}$] (1) at (1,0) {};
                        \node[label=below:$\textcolor{blue}{2}$] (2) at (2,0) {};
                        \node[label=below:$\textcolor{blue}{3}$] (3) at (3,0) {};
                        \node[label=below:$\textcolor{red}{4}$] (4) at (4,0) {};
                        \node[label=below:$\textcolor{red}{5}$] (5) at (5,0) {};
                        \node[label=below:$\textcolor{red}{6}$] (6) at (6,0) {};
                        \node[label=below:$\infty$] (0) at (7,0) {};
                        \draw[thick,color=teal]
                        (1) [out=45, in=135] to  (0);
                        \draw[thick,color=teal]
                        (3) [out=45, in=135] to  (0);
                        \draw[thick,color=teal]
                        (4) [out=45, in=135] to  (0);
                        \draw[thick,color=teal]
                        (5) [out=45, in=135] to  (0);
                        \draw[thick,color=teal]
                        (2) [out=45, in=135] to  (6);
                        \draw[thick,color=orange]
                        (6) [out=45, in=135] to  (0);
                      \end{tikzpicture}
                    \end{minipage}              &$111\cdot 112$              & 1&0&1\\
      $1|2|34|5|6$&
                   \begin{minipage}{0.45\textwidth}
                      \begin{tikzpicture}
                        [scale=1,every node/.style={circle,fill=black,inner sep=2pt, minimum size=6pt}]
                        \node[label=below:$\textcolor{blue}{1}$] (1) at (1,0) {};
                        \node[label=below:$\textcolor{blue}{2}$] (2) at (2,0) {};
                        \node[label=below:$\textcolor{blue}{3}$] (3) at (3,0) {};
                        \node[label=below:$\textcolor{red}{4}$] (4) at (4,0) {};
                        \node[label=below:$\textcolor{red}{5}$] (5) at (5,0) {};
                        \node[label=below:$\textcolor{red}{6}$] (6) at (6,0) {};
                        \node[label=below:$\infty$] (0) at (7,0) {};
                        \draw[thick,color=teal]
                        (1) [out=45, in=135] to  (0);
                        \draw[thick,color=teal]
                        (2) [out=45, in=135] to  (0);
                        \draw[thick,color=teal]
                        (5) [out=45, in=135] to  (0);
                        \draw[thick,color=teal]
                        (6) [out=45, in=135] to  (0);
                        \draw[thick,color=teal]
                        (3) [out=45, in=135] to  (4);
                        \draw[thick,color=orange]
                        (4) [out=45, in=135] to  (0);
                      \end{tikzpicture}
                    \end{minipage}&$111\cdot 211$                          & 0&2&2\\
      $1|2|35|4|6$ &
\begin{minipage}{0.45\textwidth}
                      \begin{tikzpicture}
                        [scale=1,every node/.style={circle,fill=black,inner sep=2pt, minimum size=6pt}]
                        \node[label=below:$\textcolor{blue}{1}$] (1) at (1,0) {};
                        \node[label=below:$\textcolor{blue}{2}$] (2) at (2,0) {};
                        \node[label=below:$\textcolor{blue}{3}$] (3) at (3,0) {};
                        \node[label=below:$\textcolor{red}{4}$] (4) at (4,0) {};
                        \node[label=below:$\textcolor{red}{5}$] (5) at (5,0) {};
                        \node[label=below:$\textcolor{red}{6}$] (6) at (6,0) {};
                        \node[label=below:$\infty$] (0) at (7,0) {};
                        \draw[thick,color=teal]
                        (1) [out=45, in=135] to  (0);
                        \draw[thick,color=teal]
                        (2) [out=45, in=135] to  (0);
                        \draw[thick,color=teal]
                        (4) [out=45, in=135] to  (0);
                        \draw[thick,color=teal]
                        (6) [out=45, in=135] to  (0);
                        \draw[thick,color=teal]
                        (3) [out=45, in=135] to  (5);
                        \draw[thick,color=orange]
                        (5) [out=45, in=135] to  (0);
                      \end{tikzpicture}
                    \end{minipage}       &$111\cdot 121$                     & 0&1&1\\
      $1|2|36|4|5$ &
\begin{minipage}{0.45\textwidth}
                      \begin{tikzpicture}
                        [scale=1,every node/.style={circle,fill=black,inner sep=2pt, minimum size=6pt}]
                        \node[label=below:$\textcolor{blue}{1}$] (1) at (1,0) {};
                        \node[label=below:$\textcolor{blue}{2}$] (2) at (2,0) {};
                        \node[label=below:$\textcolor{blue}{3}$] (3) at (3,0) {};
                        \node[label=below:$\textcolor{red}{4}$] (4) at (4,0) {};
                        \node[label=below:$\textcolor{red}{5}$] (5) at (5,0) {};
                        \node[label=below:$\textcolor{red}{6}$] (6) at (6,0) {};
                        \node[label=below:$\infty$] (0) at (7,0) {};
                        \draw[thick,color=teal]
                        (1) [out=45, in=135] to  (0);
                        \draw[thick,color=teal]
                        (2) [out=45, in=135] to  (0);
                        \draw[thick,color=teal]
                        (4) [out=45, in=135] to  (0);
                        \draw[thick,color=teal]
                        (5) [out=45, in=135] to  (0);
                        \draw[thick,color=teal]
                        (3) [out=45, in=135] to  (6);
                        \draw[thick,color=orange]
                        (6) [out=45, in=135] to  (0);
                      \end{tikzpicture}
                    \end{minipage}              &$111\cdot 112$              & 0&0&0\\
      \hline
    \end{tabular}
  \caption{The statistics $i_\alpha(\AA)$ and $\varphi_\alpha(\AA)$ for  $\AA\in\Pi(\mu',\alpha)$ when $\mu=(5,1)$, $\alpha=(3,3)$ and $\varphi=\inv$. Here $b_{\mu\alpha}(q)=q^4+2q^3+3q^2+2q+1$.}
  \label{tab:statistics} 
\end{table}

\begin{theorem}
  \label{thm:bviasetstat}
For each Mahonian statistic $\varphi$,
  \begin{align*}
   \tilde{b}_{\mu\alpha}(q,t)=\sum_{\AA\in \Pi(\mu',\alpha)}q^{i_\alpha(\AA)} t^{\varphi_\alpha(\AA)}.
  \end{align*}
In particular, 
    \begin{align*}
   b_{\mu\alpha}(q)=   \sum_{\AA\in \Pi(\mu',\alpha)}q^{i_{\alpha}(\AA)+\varphi_\alpha(\AA)}.
    \end{align*}

\end{theorem}
\begin{proof}
  Observe that
    \begin{align*} 
      \sum_{\AA\in \Pi(\mu',\alpha)}q^{i_\alpha(\AA)} t^{\varphi_\alpha(\AA)}                                                   &=\sum_{T\in \ssyt(\mu,\alpha)}\sum_{\substack{\hat{T}\in \syt(\mu,\alpha)\\ \SS_\alpha(\hat{T})=T}}\sum_{\substack{\AA\in \Pi(\mu',\alpha)\\ \TT(\AA)=\hat{T}}} q^{i_\alpha(\AA)}t^{\varphi_\alpha(\AA)}.
    \end{align*}
     By Proposition \ref{prop:onlysyt}, the expression on the right hand side above is equal to
    \begin{align*}
   \sum_{T\in \ssyt(\mu,\alpha)}\sum_{\substack{\hat{T}\in \syt(\mu,\alpha)\\ \SS_\alpha(\hat{T})=T}}t^{\varphi_\alpha(\hat{T})}\sum_{\substack{\AA\in \Pi(\mu',\alpha)\\ \TT(\AA)=\hat{T}}} q^{i_\alpha(\AA)}=\sum_{T\in \ssyt(\mu,\alpha)}\sum_{\substack{\hat{T}\in \syt(\mu,\alpha)\\ \SS_\alpha(\hat{T})=T}}t^{\varphi_\alpha(\hat{T})}r_q(T),    \end{align*}
 where the last equality follows from Proposition \ref{prop:interlace}. Thus 
    \begin{align*}
      \sum_{\AA\in \Pi(\mu',\alpha)}q^{i_\alpha(\AA)}t^{\varphi_\alpha(\AA)}       &=\sum_{T\in \ssyt(\mu,\alpha)}r_q(T)\sum_{\substack{\hat{T}\in \syt(\mu,\alpha)\\ \SS_\alpha(\hat{T})=T}}t^{\varphi_\alpha(\hat{T})}.
    \end{align*}
By Lemma \ref{lem:sqt}, the inner sum on the right hand side above is equal to $s_t(T)$. Therefore,
    \begin{align*}
      \sum_{\AA\in \Pi(\mu',\alpha)}q^{\varphi_\alpha(\AA)} t^{i_\alpha(\AA)}          &=\sum_{T\in \ssyt(\mu,\alpha)}r_q(T)s_t(T) =\tilde{b}_{\mu\alpha}(q,t). \qedhere
    \end{align*}
  \end{proof}

 The number of binary integer matrices with given row and column sums can be expressed in terms of $b_{\mu\alpha}(1)$ (see Corollary \ref{cor:binmat}).  

\section{Diagonal operators with a given profile}\label{sec:diagonal}
In this section, we prove that for each diagonal matrix $\Delta\in \MM_n(\Fq)$, the number of subspaces with a given $\Delta$-profile (see Definition \ref{def:profile}) can be expressed in terms of the polynomials $b_{\mu\alpha}(q)$. The main idea in this section is to use a cell decomposition of the Grassmanian along the lines of \cite[ Sec. 4]{pr}, but some new ideas are required to solve the more general recursion obtained (see Theorem \ref{thm:recurrence}) and express the answer as a sum over semistandard tableaux. Given positive integers $m$ and $n$, let $C(n,m)$ denote the set of all $m$-element subsets of $[n]$. By the pivots of a subspace $W\subseteq \Fq^n$, we mean the indices of the pivot columns of the unique matrix in row reduced echelon form whose rows span $W$. We can view the ordered set of pivots of any $m$-dimensional subspace $W\subseteq \Fq^n$ as an element of $C(n,m)$. For instance, the subspace of $\mathbb{F}_q^6$ spanned by the rows of the matrix
\begin{align*} 
  A=
  \begin{pmatrix}
    1 & 0 & 3 & 0 & 0 & 1\\
    0 & 1 & 4 & 0 & 3 & 0\\
    0 & 0 & 0 & 1 & 2& 5
  \end{pmatrix}
\end{align*}
has pivots $(1,2,4)\in C(6,3)$. To each $C\in C(n,m)$, we associate a binary word $\bb(C)$ of length $n$ defined by 
\begin{align*}
  \mathsf{b}(C):=b_1b_2\cdots b_n,
\end{align*}
where $b_i=1$ if $i\in C$ and 0 otherwise. The following result of Knuth~\cite{MR270933} shows that the number of subspaces over a finite field with pivots $C$ can be expressed in terms of the number of inversions of the word $\bb(C)$.
\begin{proposition}\label{prop:knuth}
  The number of $m$-dimensional subspaces $W\subseteq \Fq^n$ with pivots $C=(c_1,\ldots,c_m)$ is given by $q^N$ where
  \begin{align*}
N=   \inv(\bb(C))=\sum_{i=1}^m (n-m-c_i+i).
  \end{align*}
\end{proposition}
\begin{example}\label{eg:shape}
A 3-dimensional subspace of $\Fq^n(n\geq 5)$ with pivots $C=(2,4,5)$ is spanned by the rows of some matrix of the form
\begin{align*}
  \begin{pmatrix}
    0& 1 & * & 0 & 0 &* &\cdots &*\\
    0& 0 & 0 & 1 & 0 &* &\cdots &*\\
    0& 0 & 0 & 0 & 1 &* &\cdots &*
  \end{pmatrix}.
\end{align*}
Here, asterisks denote arbitrary elements of $\Fq$. Permuting the columns of this matrix by moving pivot columns to the left, we obtain the matrix
\begin{align*}
\left(  \begin{array}{ccc|ccccc}
    1& 0 & 0 & 0 & * &* &\cdots &*\\
    0& 1 & 0 & 0 & 0 &* &\cdots &*\\
    0& 0 & 1 & 0 & 0 &* &\cdots &*
  \end{array}\right),
\end{align*}
which is of the form $(I\mid X)$, where $I$ denotes the identity matrix and $X=(x_{ij})$ is a matrix with $x_{ij}=0$ for $j< c_i-(i-1)$. 
\end{example}
\begin{definition}[Shape of a matrix]
Given a composition $\beta=(\beta_1,\ldots,\beta_m)$, a matrix $X=(x_{ij})$ with $m$ rows has shape $\beta=(\beta_1,\ldots,\beta_m)$ if $x_{ij}=0$ whenever $j<\beta_i$ for $1\leq i\leq m$.
\end{definition}
Note that a given matrix can have multiple shapes. For instance, the matrix $X$ in Example~\ref{eg:shape} above has shapes $(2,3,3)$ and $(1,2,3)$. 
For each pivot set $C\in C(n,m)$ write $C-\delta$ for the tuple $(c_1,c_2-1,\ldots,c_m-(m-1))$. 

Denote by $\diag(a_1,\ldots,a_n)$ the diagonal matrix whose diagonal entries are $a_1,\ldots,a_n$. To each diagonal matrix $\Delta$, associate a set partition $\AA^\Delta\in \Pi_n$ by declaring that $i$ and $j$ lie in the same block of $\AA^\Delta$ if and only if $\Delta(i,i)=\Delta(j,j)$.
\begin{definition}\label{def:diagtype}
The type of a diagonal matrix $\Delta$ is the composition obtained by listing successive block sizes of $\AA^\Delta$ written in standard form. 
\end{definition}
For example, if $\Delta=\diag(1,4,3,4,3,4,1)$, then $\AA^\Delta=17|246|35$ and $\Delta$ has type $(2,3,2)$.
Since permuting diagonal entries does not change the conjugacy class of $\Delta$, it will occasionally be convenient to assume that the type is an integer partition.

\begin{definition}
  A square matrix is said to be \emph{block-scalar} if it is diagonal and all occurrences of each diagonal entry are contiguous.
\end{definition}
For example, the matrix $\diag(1,1,1,2,4,4)\in\MM_6(\mathbb{F}_5)$ is block-scalar but $\diag(1,2,1,1,2)$ is not. Note that every diagonal matrix is permutation similar to a block-scalar matrix. The following proposition may be viewed as a generalization of Proposition \ref{prop:knuth} as every subspace is invariant under a scalar multiple of the identity. 

\begin{proposition}\label{prop:invariant}
  Let $\Delta=\diag(a_1,\ldots,a_n)\in \MM_n(\Fq)$ be a block-scalar matrix of type $\alpha$. Given $C\in C(n,m)$, the number of $\Delta$-invariant subspaces of $\Fq^n$ with pivots $C$ is given by
  \begin{align*}
   \sigma^C(\alpha)=q^{\inv_\alpha(\bb(C))}.
  \end{align*}
\end{proposition}
\begin{proof}
  Let $W$ be an $m$-dimensional $\Delta$-invariant subspace of $\Fq^n$ with pivots $C=(c_1,\ldots,c_m)$. Let $A_W$ denote the matrix in reduced row echelon form whose rows span $W$. Moving the pivot columns of $A_W$ to the left, we obtain an $m\times n$ matrix $(I \mid X)$ where $X$ is of shape $C-\delta$. Since $W$ is $\Delta$-invariant, it follows that the rows of $(I \mid X)$ are $\tilde{\Delta}$-invariant where
  \begin{align*}
    \tilde{\Delta}=
    \begin{pmatrix}
      \Delta_1& {\bf 0}\\
     {\bf 0}&\Delta_2
    \end{pmatrix}, 
  \end{align*}
  for the diagonal matrices $\Delta_1,\Delta_2$ given by $\Delta_1=\diag(a_j:j\in C)$ and $\Delta_2=\diag(a_j:j\in [n]\setminus C)$. Thus the number of $\Delta$-invariant subspaces with pivots $C$ is equal to the number of matrices $X$ of shape $C-\delta$ for which the block matrix
  \begin{align*}
    \begin{pmatrix}
      I & X\\
      \Delta_1& X\Delta_2
    \end{pmatrix}
  \end{align*}
  has rank $m$. The above matrix is row equivalent to
  \begin{align*}
    \begin{pmatrix}
      I & X\\
      {\bf 0} & X\Delta_2-\Delta_1X
    \end{pmatrix}.
  \end{align*}
  Therefore $\sigma^C(\alpha)$ is equal to the number of $m\times (n-m)$ matrices $X$ of shape $C-\delta$ for which $X\Delta_2-\Delta_1X=0.$ If $X=(x_{ij})$, then the last condition is equivalent to $x_{ij}(\Delta_2(j,j)-\Delta_1(i,i))=0$. Since $X$ has shape $C-\delta$, we have $x_{ij}=0$ unless $j>c_i-i$. Therefore $\sigma^C(\alpha)=q^N,$ where
  \begin{align*}
    N=\sum_{\substack{i\in [m]\\j\in [n-m]\\j>c_i-i }}[\Delta_1(i,i)=\Delta_2(j,j)].
  \end{align*}
Here, for any statement $P$, the Iverson bracket $[P]$ is defined to be equal to 1 if $P$ is true and 0 otherwise. If $\phi_C:[n-m]\to [n]-C$ denotes the unique order preserving bijection, then the last sum is equivalent to
  \begin{align*}
    N=\sum_{\substack{i\in [m]\\j\in [n-m]\\\phi_C(j)>c_i }}[\Delta(c_i,c_i)=\Delta(\phi_C(j),\phi_C(j))],
  \end{align*}
  since $\phi_C(j)>c_i$ if and only if $j>c_i-i$. Since $\Delta(\phi_C(j),\phi_C(j))=\Delta(c_i,c_i)$ precisely when $\phi_C(j)$ and $c_i$ lie in the same block of $\AA_\alpha$, it follows that $N$ is equal to the number of pairs $(c,c')\in C\times ([n]\setminus C)$ such that $c<c'$ and both $c,c'$ lie in the same block of $\AA_\alpha$. Therefore $N=\inv_\alpha(\bb(C))$.  
\end{proof}

\begin{corollary}\label{cor:fiber}
  If $\Delta_1,\Delta_2$ are as in the proof of Proposition \ref{prop:invariant}, then each nonempty fiber of the map $X\mapsto X\Delta_2-\Delta_1X$ defined on $m\times (n-m)$ matrices of shape $C-\delta$ has cardinality $q^{\inv_\alpha(\bb(C))}$. 
\end{corollary}

 Let $\PP$ denote the set of positive integers. For each pivot set $C$, let $\phi_C:\PP \to \PP \setminus C$ be the unique order preserving bijection. Given pivot sets $C$ and $D$, denote by $T(C,D)$ the tableau whose rows are given by the elements of $C$ and $\phi_C(D)$ arranged in increasing order. For example, when $C=(1,3,4)$ and $D=(1,3)$, we have  
\begin{align*}
  T(C,D)=
  \begin{ytableau} 
      1&3&4\\
      2&6
  \end{ytableau}.
\end{align*}
 Note that the rows of $T(C,D)$ are increasing by construction, but this is not necessarily true of the columns. For instance, when $C=(1,4,5)$ and $D=(1,2,4)$, we have
\begin{align*}
  T(C,D)=
  \begin{ytableau}
      1&4&5\\
      2&3&7
  \end{ytableau},
\end{align*}
which does not have increasing columns. Our next objective is to derive a recursion for the number of subspaces with pivots $C$ and $\Delta$-profile $\mu$. We require the following lemma.

\begin{lemma}\label{lem:eta}
  Let $\mu$ be an integer partition and suppose $\Delta=\diag(a_1,\ldots,a_n)\in \MM_n(\Fq)$ is a block-scalar matrix of type $\alpha$. Given $C\in C(n,\mu_1)$ and $D\in C(n-\mu_1,\mu_2)$, define diagonal matrices $\Delta_1,\Delta_2$ by $\Delta_1=\diag(a_j:j\in C)$ and $\Delta_2=\diag(a_j:j\in [n]\setminus C)$. The number of $\mu_1\times (n-\mu_1)$ matrices $Y$ of shape $C-\delta$ and of the form $Y=X\Delta_2-\Delta_1X$ whose rows span a given subspace $W\subseteq \Fq^{n-\mu_1}$ with pivots $D$ is given by
  \begin{align*}
    \eta_\alpha(C,D)=(q-1)^{\mu_2}q^{\mu_2 \choose 2}r_q(\SS_\alpha(T(C,D))). 
  \end{align*}
\end{lemma}
\begin{proof}
  Since $Y=(y_{ij})$ has shape $C-\delta$ it follows that $y_{ij}=0$ unless $j>c_i-i$ which is equivalent to $\phi_C(j)>c_i$. Moreover, since $Y=X\Delta_2-\Delta_1X$, it follows that $x_{ij}(\Delta_2(j,j)-\Delta_1(i,i))=y_{ij}.$ Thus $y_{ij}=0$ whenever $\Delta_2(j,j)=\Delta_1(i,i)$, or equivalently, $\Delta(\phi_C(j),\phi_C(j))=\Delta(c_i,c_i).$ The last equality occurs if and only if $c_i$ and $\phi_C(j)$ lie in the same block of $A_\alpha$, which is equivalent to $\SS_\alpha(c_i)=\SS_\alpha(\phi_C(j))$. Combining the above observations, it follows that $y_{ij}=0$ unless $\SS_\alpha(\phi_C(j))>\SS_\alpha(c_i)$.
  
  Let $E$ denote the unique $\mu_2\times (n-\mu_1)$ matrix in reduced row echelon form whose rows span $W$. Then we can write $Y=AE$ for some $\mu_1 \times \mu_2 $ matrix $A=(a_{ij})$ of rank $\mu_2$. To count the number of possible choices for $Y$, we count the number of choices for $A$ by successively choosing its columns. For $1\leq i\leq \mu_1$, the $i$th row of $Y$ is given by 
  \begin{align*}
    Y_i=\sum_{j=1}^{\mu_2}a_{ij}E_j,
  \end{align*}
where $E_j$ is the $j$th row of $E$ for $j\geq 1$. If we write $D=(d_1,d_2,\ldots)$ then, since $E$ has pivots $D$, it follows that the first nonzero entry of $E_j$ is in position $d_j$. Now $a_{ij}\neq 0$ if and only if $y_{i, d_j}\neq 0$ which can only happen when $\SS_\alpha(\phi_C(d_j))>\SS_\alpha(c_i)$ (see Example \ref{eg:eta}). Therefore, the number of potentially nonzero elements in the $j$th column of $A$ is equal to 
  \begin{align*}
    \#\{i:\SS_\alpha(\phi_C(d_j))>\SS_\alpha(c_i) \}=r_{2j}(\SS_\alpha(T(C,D)))+j-1.
  \end{align*}
  Writing $r_{2j}$ for $r_{2j}(\SS_\alpha(T(C,D))) $, the number of choices for the first column of $A$ is given by $q^{r_{21}}-1$. The second column of $A$ is linearly independent of the first column and can be chosen in $q^{r_{22}+1}-q$ ways. More generally, once the first $j-1$ columns of $A$ have been chosen, the $j$th column can be chosen in $q^{r_{2j}+j-1}-q^{j-1}$ ways. It follows that the number of choices for $A$ is equal to
  \begin{align*}
    \prod_{j=1}^{\mu_2}(q^{r_{2j}+j-1}-q^{j-1})&=(q-1)^{\mu_2}q^{\mu_2 \choose 2}\prod_{j=1}^{\mu_2}[r_{2j}]_q\\
    &=(q-1)^{\mu_2}q^{\mu_2 \choose 2}r_q(\SS_\alpha(T(C,D))).\qedhere
  \end{align*}
\end{proof}
\begin{example}\label{eg:eta}
  Suppose $n=12$ while $\mu=(5,4,3)$ and $\alpha=(2,3,2,2,3)$. Let $C=(1,2,4,7,8)$ and $D=(1,2,3,5)$. In this case, 
  \begin{align*}
    T(C,D)=
    \begin{ytableau}
      1&2&4&7&8\\
      3&5&6&10 
    \end{ytableau}, \qquad
\SS_\alpha(T(C,D))=\begin{ytableau}
      1&1&2&3&4\\
      2&2&3&5 
    \end{ytableau}.
  \end{align*}
  Here $A$ is a $5\times 4$ matrix of the form  
  \begin{align*}
    A=
    \begin{pmatrix}
      * & * & * & *\\
      * & * & * & *\\
      0 & 0 & * & *\\
      0 & 0 & 0 & *\\
      0 & 0 & 0 & *
    \end{pmatrix}
  \end{align*}
  Observe that the number of zeros in the $i$th row of $A$ is equal to the number of entries in the second row of the tableau $\SS_\alpha(T(C,D))$ less than or equal to the $i$th entry in the first row. Therefore the number of potentially nonzero entries in successive rows of $A$ is weakly decreasing. The number of choices for $A$ having full column rank is therefore equal to
  \begin{align*}
 (q^2-1)(q^2-q)(q^3-q^2)(q^5-q^3)=  \eta_\alpha(C,D).
  \end{align*}
\end{example}

\begin{remark}
Note that $r_q(\SS_\alpha(T(C,D)))$ (and hence $\eta_\alpha(C,D)$) is zero whenever $\SS_\alpha(T(C,D))$ does not have increasing columns. In particular, this happens when $T(C,D)$ itself does not have increasing columns.
\end{remark}
For each pivot set $C$ and integer partition $\mu$, let $\sigma^C_n(\mu,\Delta)$ denote the number of subspaces with pivots $C$ and $\Delta$-profile $\mu$. In the following theorem we obtain a recursion satisfied by $\sigma^C_n(\mu,\Delta)$. 
\begin{theorem}\label{thm:recurrence}
Let $\mu$ be an integer partition and suppose $\Delta=\diag(a_1,\ldots,a_n)$ is a block-scalar matrix over $\Fq$ of type $\alpha$. For each pivot set $C\in C(n,\mu_1)$, let $\Delta_C$ denote the diagonal matrix defined by $\Delta_C=\diag(a_j:j\in [n]\setminus C)$. We have
  \begin{align*}
    \sigma^C_n(\mu,\Delta)=q^{\inv_{\alpha}(\bb(C))}\sum_{D\in C(n-\mu_1,\mu_2)} \eta_\alpha(C,D) \sigma^D_{n-\mu_1}(\tilde{\mu},\Delta_C),
  \end{align*}
  where $\tilde{\mu}=(\mu_2,\mu_3,\ldots)$. 
\end{theorem}
\begin{proof}
  Let $W$ be a subspace of $\Fq^n$ with pivots $C=(c_1,\ldots,c_{\mu_1})$. Consider the $\mu_1\times n$ matrix in reduced row echelon form whose rows span $W$ and move its pivot columns to the left to obtain a matrix $A=(I \mid X)$ where $X$ is a $\mu_1\times (n-\mu_1)$ matrix of shape $C-\delta$. The rows of $A$ span a subspace $\widetilde{W}\subseteq \Fq^n$. Now $W$ has $\Delta$-profile $\mu$ if and only if $\widetilde{W}$ has $\tilde{\Delta}$-profile $\mu$ where $\tilde{\Delta}$ is the diagonal matrix given by 
  \begin{align*}
    \tilde{\Delta}=
    \begin{pmatrix}
      \Delta_1& {\bf 0}\\
     {\bf 0}&\Delta_2
    \end{pmatrix}, 
  \end{align*}
  with $\Delta_1=\diag(a_i:i\in C)$ and $\Delta_2=\diag(a_i:i\in [n]\setminus C)$. Write $\mu=(\mu_1,\mu_2,\ldots)$ and observe that $\widetilde{W}$ has $\tilde{\Delta}$-profile $\mu$ precisely when the block matrix
  \begin{align*}
 \begin{pmatrix}
      I & X\\
      \Delta_1 & X\Delta_2\\
      \vdots &\vdots\\
      \Delta_1^{j-1} & X\Delta_2^{j-1}
      \end{pmatrix}
  \end{align*}
  has rank $\mu_1+\cdots+\mu_j$ for each $j\geq 1$. Applying the block row operation $R_i\to R_i-\Delta_1^{i-1}R_1$ successively for $i=2,3,\ldots,j,$ we obtain the matrix
  \begin{align*}
  \begin{pmatrix}
      I & X\\
      {\bf 0} & X\Delta_2-\Delta_1X\\
      \vdots &\vdots\\
      {\bf 0} & X\Delta_2^{j-1}-\Delta_1^{j-1}X
      \end{pmatrix}.
  \end{align*}
  The rank condition above implies that the matrix
  \begin{align}\label{eq:block}
  \begin{pmatrix}
    X\Delta_2-\Delta_1X\\
\vdots\\
X\Delta_2^{j-1}-\Delta_1^{j-1}X
      \end{pmatrix}
  \end{align}
  has rank $\mu_2+\cdots+\mu_j$ for each $j\geq 2$. If $Y=X\Delta_2-\Delta_1X$, then it is easily verified that
  \begin{align*}
        X\Delta_2^i - \Delta_1^i X = Y\Delta_2^{i-1} + \Delta_1Y\Delta_2^{i-2} + \cdots + \Delta_1^{i-2}Y\Delta_2 + \Delta_1^{i-1}Y.
  \end{align*}
Thus the matrix \eqref{eq:block} can be written as 
  \begin{align*}
    \begin{pmatrix}
      Y\\
      Y\Delta_2+\Delta_1Y\\
      \vdots\\
      Y\Delta_2^{j-2} + \Delta_1Y\Delta_2^{j-3} + \cdots + \Delta_1^{j-3}Y\Delta_2 + \Delta_1^{j-2}Y
    \end{pmatrix}.
  \end{align*}
Now apply the block row operation $R_i\to R_i- \Delta_1R_{i-1}$ successively for $i=j-1,j-2,\ldots,2$ to the matrix above to obtain
  \begin{align*}
    \begin{pmatrix}
      Y\\
      Y\Delta_2\\
      \vdots\\
      Y\Delta_2^{j-2}
      \end{pmatrix}.
  \end{align*}
  The rank condition implies that the matrix above has rank $\mu_2+\cdots+\mu_j$ for $j\geq 2$. It follows that the row space of $Y$ has $\Delta_2$-profile $\tilde{\mu}=(\mu_2,\mu_3,\ldots)$. Since $X$ has shape $C-\delta$, it follows that $Y$ also has shape $C-\delta$. By Corollary \ref{cor:fiber}, the number of $\mu_1\times (n-\mu_1)$ matrices $X$ of shape $C-\delta$ such that $X\Delta_2-\Delta_1X=Y$ for a fixed matrix $Y$ is equal to $q^{\inv_\alpha(\bb(C))}.$ Consider a pivot set $D\in C(n-\mu_1,\mu_2).$ Lemma \ref{lem:eta} implies that the number of $\mu_1\times (n-\mu_1)$ matrices $Y$ of shape $C-\delta$ and of the form $Y=X\Delta_2-\Delta_1X$ whose rows span a given subspace of $\Fq^{n-\mu_1}$ with pivots $D$ is given by $\eta_\alpha(C,D)$. Putting these observations together, it follows that
  \begin{align*}
        \sigma^C_n(\mu,\Delta)=q^{\inv_{\alpha}(\bb(C))}\sum_{D\in C(n-\mu_1,\mu_2)} \eta_\alpha(C,D) \sigma^D_{n-\mu_1}(\tilde{\mu},\Delta_2),
  \end{align*}
 where $\Delta_2=\diag(a_j:j\in [n]\setminus C)=\Delta_C$. This completes the proof.
\end{proof}
The recursion in the theorem above can be solved to express $\sigma_n^C(\mu,\Delta)$ as a sum over a class of tableaux which we now define. Given an integer partition $\mu$, a \emph{multilinear tableau} of shape $\mu$ is a filling of the Young diagram of $\mu$ with positive integers such that each integer appears at most once. Let ${\rm Tab}_{\subseteq [n]}(\mu)$ denote the set of all multilinear tableaux of shape $\mu$ with increasing rows and columns and with entries in $[n]$. We have the following extension of Definition \ref{def:tabtoword}.
\begin{definition}\label{def:generalword}
To each tableau $T\in \tabnmu$ with $r$ rows, we associate a word 
\begin{align*}
  w(T):=a_1a_2\cdots a_n,
\end{align*}
where
\begin{align*}
a_i=\begin{cases}
      j & \mbox{ if the entry $i$ appears in row $j$ of } T\\
      r+1 & \mbox{ otherwise. }
\end{cases}
\end{align*}
  \end{definition}
For instance, when $n=8$, $\mu=(3,2)$ and
\begin{align*}
T=\begin{ytableau}
  1& 3 & 5\\
    2 & 6
  \end{ytableau},
\end{align*}
we have $w(T)=12131233.$
\begin{definition}
The number of \emph{non-inversions} of a word $w=w_1w_2\cdots w_n$ is defined as the number of pairs $1\leq i<j\leq n$ such that $w_i<w_j$.  Denote by $\ninv(w)$ the number of non-inversions of $w$. 
\end{definition}
It is well-known that $\ninv$ is a Mahonian statistic. 

\begin{theorem}\label{thm:sigmac}
  Let $\Delta$ be a block-scalar $n\times n$ matrix of type $\alpha$ and suppose $\mu$ is an integer partition. For each pivot set $C\in C(n,\mu_1)$,
  \begin{align*}
    \sigma_n^C(\mu,\Delta)=(q-1)^{\sum_{j\geq 2}\mu_j}q^{\sum_{j\geq 2}{\mu_j \choose 2}} \sum_{\substack{T\in \tabnmu \\ T \text{ has first row } C}} q^{\ninv_\alpha(w(T))}r_q(\salpha(T)). 
  \end{align*}
\end{theorem}
\begin{proof}
  Induct on the number of parts $\ell$ of $\mu$. When $\ell=1$, we have $\mu=(m)$ for some positive integer $m$. By Proposition \ref{prop:invariant}, we have $\sigma_n^C((m),\Delta)=q^{\inv_\alpha(\bb(C))}$. If $T$ denotes the tableau with a single row whose entries are given by $C$, then $w(T)$ can be obtained from $\bb(C)$ by replacing all the zeros by $2$. It follows that $\inv_\alpha(\bb(C))=\ninv_\alpha(w(T))$. Moreover, since $r_q(\salpha(T))=1,$ the theorem holds for $\ell=1$.

  Suppose the result holds when $\mu$ has at most $\ell-1$ parts. Let $\mu=(\mu_1,\ldots,\mu_\ell)$ and suppose $\Delta=\diag(a_1,\ldots,a_n)$ is a block-scalar matrix of type $\alpha=(\alpha_1,\ldots,\alpha_r)$. By Theorem~\ref{thm:recurrence}, 
  \begin{align*}
      \sigma^C_n(\mu,\Delta)=q^{\inv_{\alpha}(\bb(C))}\sum_{D\in C(n-\mu_1,\mu_2)} \eta_\alpha(C,D) \sigma^D_{n-\mu_1}(\tilde{\mu},\tilde{\Delta}),
  \end{align*}
  where $\tilde{\mu}=(\mu_2,\mu_3,\ldots)$ and $\tilde{\Delta}$ is the block-scalar matrix $\diag(a_j:j\in [n]\setminus C)$. Let $\tilde{\alpha}$ denote the type of $\tilde{\Delta}$. By the induction hypothesis, 
  \begin{align*}
    \sigma^D_{n-\mu_1}(\tilde{\mu},\tilde{\Delta})=(q-1)^{\sum_{j\geq 3}\mu_j}q^{\sum_{j\geq 3}{\mu_j \choose 2}} \sum_{\substack{\tilde{T}\in {\rm Tab}_{\subseteq [n-\mu_1]}(\tilde{\mu}) \\ \tilde{T} \text{ has first row } D}} q^{\ninv_{\tilde{\alpha}}(w(\tilde{T}))} r_q(\SS_{\tilde\alpha}(\tilde{T})).
  \end{align*}
  Using the fact that
    \begin{align*}
    \eta_\alpha(C,D)=(q-1)^{\mu_2}q^{\mu_2 \choose 2}r_q(\SS_\alpha(T(C,D))),
  \end{align*}
  we obtain
  \begin{align}
    \sigma^C_n(\mu,\Delta)&=(q-1)^{\sum_{j\geq 2}\mu_j}\cdot q^{{\inv_{\alpha}(\bb(C))}+\sum_{j\geq 2}{\mu_j \choose 2}}\times  \label{eq:bigone}\\ 
    &\quad\sum_{D\in C(n-\mu_1,\mu_2)}r_q(\SS_\alpha(T(C,D))) \sum_{\substack{\tilde{T}\in {\rm Tab}_{\subseteq [n-\mu_1]}(\tilde{\mu}) \\ \tilde{T} \text{ has first row } D}} q^{\ninv_{\tilde{\alpha}}(w(\tilde{T}))} r_q(\SS_{\tilde\alpha}(\tilde{T})). \nonumber
  \end{align}
  Let $\phi_C:\PP\to \PP\setminus C$ denote the unique order preserving bijection. For each tableau $\tilde{T}\in {\rm Tab}_{\subseteq [n-\mu_1]}(\tilde{\mu}),$ let $\phi_C(\tilde{T})$ denote the tableau obtained by replacing each entry $x$ in $\tilde{T}$ by $\phi_C(x)$. Thus, $\phi_C$ induces a map (also denoted $\phi_C$ by a slight abuse of notation):
  \begin{align*}
    \phi_C:{\rm Tab}_{\subseteq [n-\mu_1]}(\tilde{\mu})\to {\rm Tab}_{\subseteq [n]\setminus C}(\tilde{\mu}).
  \end{align*}
  Consider the weak composition $\hat{\alpha}=(\hat{\alpha}_1,\ldots,\hat{\alpha}_r)$ obtained by padding $\tilde{\alpha}$ with zeros (if necessary) such that $\alpha_j$ and $\hat{\alpha}_j$ correspond to the multiplicities of the same element of $\Fq$ for $1\leq j\leq r$. Let $\SS_{\hat{\alpha}}(x)$ denote the least integer $i$ for which $x\leq \hat{\alpha}_1+\cdots+\hat{\alpha}_i$. We claim that $\SS_{\hat{\alpha}}(x)=\salpha(\phi_C(x))$ for each $x\in [n-\mu_1]$. This follows from the easily verified fact that $x\leq \hat{\alpha}_1+\cdots+\hat{\alpha}_i$ if and only if $\phi_C(x)\leq \alpha_1+\cdots+\alpha_i$ for each $i\geq 1$. Therefore, $\SS_{\hat{\alpha}}(\tilde{T})=\salpha(\phi_C(\tilde{T})).$ We may not have $\SS_{\tilde{\alpha}}(x)=\SS_{\hat{\alpha}}(x)$ but, since $\hat{\alpha}$ is obtained by padding $\tilde{\alpha}$ with zeros, it is easily seen that $\SS_{\tilde{\alpha}}(x)<\SS_{\tilde{\alpha}}(y)$ if and only if $\SS_{\hat{\alpha}}(x)<\SS_{\hat{\alpha}}(y).$ Therefore, $r_q(\SS_{\tilde{\alpha}}(\tilde{T}))=r_q(\SS_{\hat{\alpha}}(\tilde{T}))$ for each tableau $\tilde{T}\in {\rm Tab}_{\subseteq [n-\mu_1]}(\tilde{\mu}).$

  Since $\inv_\alpha(\bb(C))$ equals the value taken by $\ninv_\alpha$ on the word corresponding to a tableau with a single row with entries $C$, it can be seen that
  \begin{align*}
    \inv_\alpha(\bb(C))+\ninv_{\tilde{\alpha}}(w(\tilde{T}))=\ninv_\alpha(w(T)),
  \end{align*}
  where $T=[C,\phi_C(\tilde{T})]$ denotes the tableau whose first row is $C$ and subsequent rows are those of $\phi_C(\tilde{T})$. Since the rows of $T(C,D)$ are $C$ and $\phi_C(D)$, it follows that
  \begin{align*}
    r_q(\salpha(T(C,D))) \cdot r_q(\SS_{\tilde{\alpha}}(\tilde{T}))&=   r_q(\salpha(T(C,D))) \cdot r_q(\SS_{\hat{\alpha}}(\tilde{T})) \\
    &= r_q(\salpha(T(C,D))) \cdot r_q(\SS_{\alpha}(\phi_C(\tilde{T})))\\
    &=r_q(\salpha([C,\phi_C(\tilde{T})])), 
  \end{align*}
whenever the first row of $\tilde{T}$ is equal to $D$. Abbreviating $[C,\phi_C(\tilde{T})]$ to $T$, we have by \eqref{eq:bigone} 
  \begin{align*}
    \sigma^C_n(\mu,\Delta)&=(q-1)^{\sum_{j\geq 2}\mu_j}\cdot q^{\sum_{j\geq 2}{\mu_j \choose 2}} \sum_{D\in C(n-\mu_1,\mu_2)} \sum_{\substack{\tilde{T}\in {\rm Tab}_{\subseteq [n-\mu_1]}(\tilde{\mu}) \\ \tilde{T} \text{ has first row } D}} q^{\ninv_{\alpha}(w(T))} r_q(\salpha(T))\\
                          &=(q-1)^{\sum_{j\geq 2}\mu_j}\cdot q^{\sum_{j\geq 2}{\mu_j \choose 2}}  \sum_{\tilde{T}\in {\rm Tab}_{\subseteq [n-\mu_1]}(\tilde{\mu})} q^{\ninv_{\alpha}(w(T))} r_q(\salpha(T))\\
    &=(q-1)^{\sum_{j\geq 2}\mu_j}\cdot q^{\sum_{j\geq 2}{\mu_j \choose 2}}\sum_{T'\in {\rm Tab}_{\subseteq [n]\setminus C}(\tilde{\mu})} q^{\ninv_{\alpha}(w(\hat{T}))} r_q(\salpha(\hat{T})),
  \end{align*}
  where $\hat{T}$ is the tableau whose first row is $C$ and other rows are those of $T'.$ It follows that
  \begin{align*}
       \sigma_n^C(\mu,\Delta)&=(q-1)^{\sum_{j\geq 2}\mu_j}q^{\sum_{j\geq 2}{\mu_j \choose 2}} \sum_{\substack{T\in \tabnmu \\ T \text{ has first row } C}} q^{\ninv_\alpha(w(T))}r_q(\salpha(T)),
  \end{align*}
  completing the proof.
\end{proof}
\begin{definition}
Given a composition $\alpha$ and a partition $\mu$, let $\tabnmualpha$ denote the set of all tableaux $T\in \tabnmu$ such that no column of $T$ contains more than one element of the canonical set partition $\AA_\alpha$.    
\end{definition}

\begin{corollary}\label{cor:sigmudelt}
  If $\Delta$ is a diagonal matrix of type $\alpha$ and $\mu$ is an integer partition, then the number of subspaces with $\Delta$-profile $\mu$ is given by
  \begin{align*}
    \sigma(\mu,\Delta)=(q-1)^{\sum_{j\geq 2}\mu_j}q^{\sum_{j\geq 2}{\mu_j \choose 2}}\sum_{T\in \tabnmualpha} q^{\ninv_\alpha(w(T))}r_q(\salpha(T)).
  \end{align*}
\end{corollary}
\begin{proof}
If $\Delta$ is block-scalar, the result follows from Theorem \ref{thm:sigmac} by summing over all possible pivot sets $C$ since $r_q(\salpha(T))=0$ for $T\in \tabnmu \setminus \tabnmualpha$. The result now follows for arbitrary diagonal matrices since each such matrix is similar to a block-scalar matrix.
\end{proof}
We will see that the expression obtained for $\sigma(\mu,\Delta)$ above can be expressed more compactly as a sum over certain semistandard tableaux.
\begin{definition}
Given a composition $\alpha=(\alpha_1,\alpha_2,\ldots)$ of $n$, let ${\rm SSTab(\mu,\alpha)}$ denote the set of all fillings of the Young diagram of $\mu$ by submultisets of $\{1^{\alpha_1},2^{\alpha_2},\ldots\}$ such that the rows are weakly increasing from left to right while the columns are strictly increasing from top to bottom.   
\end{definition}

\begin{example}
  If $\mu=(2,1)$ and $\alpha=(2,1,2),$ then $\sstabmua$ consists of column-strict fillings of the Young diagram of $(2,1)$ with a submultiset of $\{1,1,2,3,3\}$. There are six such tableaux:
  \begin{align*}
    \begin{ytableau}
      1 & 1\\
      2
    \end{ytableau}, 
        \begin{ytableau}
      1 & 1\\
      3
    \end{ytableau}, 
    \begin{ytableau}
      1 & 2\\
      3
    \end{ytableau}, 
    \begin{ytableau}
      1 & 3\\
      2
    \end{ytableau}, 
    \begin{ytableau}
      1 & 3\\
      3
    \end{ytableau},
        \begin{ytableau}
      2 & 3\\
      3
    \end{ytableau}.
  \end{align*}
\end{example}
\begin{definition}
For $T\in \sstabmua$, define
\begin{align*}
  s_q^{\alpha}(T):=\prod_{i\geq 1}{\alpha_i \brack \alpha_i-|\beta^i|,\beta^i_{1},\beta^i_{2},\ldots}_q,
\end{align*}
where $\beta^i$ and $\beta^i_j$ are as in Definition \ref{def:bit}.    
\end{definition}
Note that when $|\mu|=|\alpha|$, we have $\alpha_i=|\beta^i|$; therefore $s_q^\alpha(T)$ is simply $s_q(T)$ as in Definition~\ref{def:sqt}. We now extend Definition \ref{def:bmualpha} for  $b_{\mu\alpha}(q)$ to all partitions $\mu$. 
\begin{definition}\label{def:bmualphaall}
  Define
  \begin{align*}
    b_{\mu\alpha}(q):=\sum_{T\in \sstabmua}r_q(T)s_q^\alpha(T).
  \end{align*}
\end{definition}

Theorem \ref{thm:profilesviassyt} gives a formula for the number of subspaces with a given profile in terms of the polynomials $b_{\mu\alpha}(q)$. The theorem rests on the following lemma.  
\begin{lemma}\label{lem:sqalpha}
 For each composition $\alpha$  of $n$, and each tableau $T\in \sstabmua$, 
  \begin{align*}
    \sum_{\substack{\hat{T}\in \tabnmualpha\\ \salpha(\hat{T})=T} }q^{\ninv_\alpha(\hat{T})}=s_q^{\alpha}(T).
  \end{align*}
\end{lemma}
\begin{proof}
  The proof is very similar to that of Lemma \ref{lem:sqt}; we give a sketch of the proof. For each tableau $T\in {\rm Tab}_{\subseteq [n]}(\mu)$, consider the associated word $w(T)=a_1a_2\cdots a_n$ as in Definition \ref{def:generalword}. Suppose $\alpha=(\alpha_1,\ldots,\alpha_k)$ and that $\mu$ has $\ell$ parts. By an argument analogous to the one in the proof of Proposition \ref{prop:bijection},  we obtain
  \begin{align*}
    \sum_{\substack{\hat{T}\in \syt(\mu,\alpha)\\ \SS_\alpha(\hat{T})=T}}q^{\ninv_\alpha(w(\hat{T}))}&=\sum_{(w_1,\ldots,w_k)\in \prod_{i\geq 1}R(\gamma^i)} q^{\ninv(w_1)+\cdots+\ninv(w_k)},
  \end{align*}
where $\gamma^i$ denotes the weak composition $({\beta^i_1},{\beta^i_2},\ldots, {\beta^i_{\ell}},{\alpha_i-|\beta^i(T)|})$. Write $\beta^i=(\beta^i_1,\beta^i_2,\ldots)$. Since $\ninv$ is Mahonian, the sum on the right hand side above equals
  \begin{align*}
         \prod_{i\geq 1}\sum_{w_i \in R(\gamma^i)}q^{\ninv(w_i)} &=\prod_{i\geq 1}{\alpha_i \brack \alpha_i-|\beta^i|,\beta^i_{1},\beta^i_{2},\ldots}_q =s^\alpha_q(T). \qedhere
  \end{align*}
\end{proof}

\begin{theorem}\label{thm:profilesviassyt}
 Let $\Delta$ be a diagonal matrix of type $\alpha$ over $\Fq$.   The number of subspaces with $\Delta$-profile $\mu$ is given by
  \begin{align*}
    \sigma(\mu,\alpha)=(q-1)^{\sum_{j\geq 2}\mu_j}q^{\sum_{j\geq 2}{\mu_j \choose 2}}b_{\mu\alpha}(q).
  \end{align*}
\end{theorem}

\begin{proof} 
 Note that if $\hat{T}\in \tabnmualpha$, then $\salpha(\hat{T})\in \sstabmua.$ By Corollary \ref{cor:sigmudelt}, 
  \begin{align*}
    \sigma(\mu,\Delta)&=(q-1)^{\sum_{j\geq 2}\mu_j}q^{\sum_{j\geq 2}{\mu_j \choose 2}}\sum_{\hat{T}\in \tabnmualpha} q^{\ninv_\alpha(w(\hat{T}))}r_q(\salpha(\hat{T}))  \\
                      &=(q-1)^{\sum_{j\geq 2}\mu_j}q^{\sum_{j\geq 2}{\mu_j \choose 2}}\sum_{T\in \sstabmua}     \sum_{\substack{\hat{T}\in \tabnmualpha\\ \salpha(\hat{T})=T} } q^{\ninv_\alpha(w(\hat{T}))}r_q(\salpha(\hat{T}))  \\
                      &=(q-1)^{\sum_{j\geq 2}\mu_j}q^{\sum_{j\geq 2}{\mu_j \choose 2}}\sum_{T\in \sstabmua}  r_q(T)   \sum_{\substack{\hat{T}\in \tabnmualpha\\ \salpha(\hat{T})=T} } q^{\ninv_\alpha(w(\hat{T}))} .
  \end{align*}
  By Lemma \ref{lem:sqalpha}, the innermost sum is simply $s_q^{\alpha}(T)$ and the theorem follows.
\end{proof}

Observe that $\sigma(\mu,\alpha)$ is invariant under permutations of the coordinates of $\alpha$, since permuting the entries of a diagonal matrix leaves its conjugacy class invariant. Consequently, $b_{\mu\alpha}(q)$ is invariant under permutations of coordinates of $\alpha$. Therefore it suffices to study $b_{\mu\nu}(q)$ for integer partitions $\nu$. Another point worth noting is that one can define $b_{\mu\alpha}(q)$ by treating $q$ as a formal variable using Definition \ref{def:bmualphaall} even when there is no diagonal matrix of type $\alpha$ over the finite field $\Fq$. When $q$ is treated as a formal variable, it is clear that $b_{\mu\alpha}(q)$ is a polynomial in $q$ and we assume this viewpoint in the next section.
\begin{remark}
  The results in this section play a crucial role in the solution to the general case of Problem~\ref{prob:main} for arbitrary operators. The reader is referred to \cite{ram2023subspace} for the details.  
\end{remark}

\section{$q$-Whittaker functions and the Touchard--Riordan formula}    
\label{sec:symmetric}
In this section we show that the polynomials $b_{\mu\nu}(q)$ are closely related to coefficients in the monomial expansion of the $q$-Whittaker symmetric function. Though the name `$q$-Whittaker' was coined by Gerasimov, Lebedev and Oblezin \cite{MR2575477}, they had already been considered previously by Macdonald~\cite{MR1354144}. The $q$-Whittaker function $W_\lambda({\bf x};q)$ is obtained as the $t=0$ specialization of the Macdonald polynomial $P_\lambda({\bf x};q,t)$. One has the specializations  
\begin{align}\label{eq:qwspecialization}
  W_\lambda({\bf x};0)=s_{\lambda}({\bf x}), \quad W_\lambda({\bf x};1)=e_{\lambda'}({\bf x}),
\end{align}  
where $s_\lambda$ denotes the Schur function and $e_\lambda$  denotes the elementary symmetric function. The $q$-Whittaker function arises in a representation theoretic setting in the context of the graded Frobenius characteristic of the cohomology ring of Springer fibers. The reader is referred to the survey article of Bergeron \cite{bergeron2020survey} for more on $q$-Whittaker functions. Denoting by $m_\lambda$ the monomial symmetric function indexed by $\lambda$, we have (Macdonald \cite[VI.(7.13')]{MR1354144}) 
\begin{align} 
  W_{\mu}(\bx)= \sum_{\nu}a_{\mu\nu}(q)m_\nu,\label{eq:defofas}
\end{align}
where, for each partition $\nu=(\nu_1,\ldots,\nu_k)$, 
\begin{align*}
  a_{\mu\nu}(q)=\sum \prod_{j=1}^k \psi_{\mu^{j}/\mu^{j-1}}(q),%\sum_{\substack{\emptyset=\mu^{0}\subseteq \mu^{1}\subseteq \cdots\subseteq \mu^{k}=\mu\\|\mu^{j}/\mu^{j-1}|=\nu_j} }
\end{align*}
where the sum is taken over all chains $\emptyset=\mu^{0}\subseteq \mu^{1}\subseteq \cdots\subseteq \mu^{k}=\mu$  of partitions such that $\mu^{j}/\mu^{j-1}$ is a horizontal $\nu_j$-strip. Here
\begin{align*}
  \psi_{\mu/\rho}(q)=\prod_{i\geq 1}{\mu_i-\mu_{i+1} \brack \mu_i-\rho_i}_q.
\end{align*}
We will show in Theorem \ref{thm:whittakerconnection} that the the polynomials $b_{\mu\nu}(q)$ can be expressed in terms of the coefficients $a_{\mu\nu}(q)$ above.  
\begin{proposition}
Given partitions $\mu$ and $\nu=(\nu_1,\ldots,\nu_k)$ of the same size,
\begin{align*}
  b_{\mu\nu}(q)= \sum\prod_{j=1}^k \theta_{\mu^{j}/\mu^{j-1}}(q),%\sum_{\substack{\emptyset=\mu^{0}\subseteq \cdots\subseteq \mu^{k}=\mu\\|\mu^{j}/\mu^{j-1}|=\nu_j} }
\end{align*}
where the sum is taken over all chains $\emptyset=\mu^{0}\subseteq \mu^{1}\subseteq \cdots\subseteq \mu^{k}=\mu$ of partitions such that $\mu^{j}/\mu^{j-1}$ is a horizontal $\nu_j$-strip ($1\leq j\leq k$) and 
\begin{align*}
  \theta_{\mu/\rho}(q)=\frac{[|\mu|-|\rho|]_q!}{[\mu_1-\rho_1]_q!}\prod_{i\geq 1}{\rho_i-\rho_{i+1} \brack \mu_{i+1}-\rho_{i+1}}_q,
\end{align*}
where $[k]_q!$ denotes the product $[1]_q[2]_q\cdots [k]_q$.
\end{proposition}

\begin{proof}
By Definition \ref{def:bmualpha},
  \begin{align*}
    b_{\mu\nu}(q)=\sum_{T\in \ssyt(\mu,\nu)}r_q(T)s_q(T).
  \end{align*}
Consider a tableau $T\in \ssyt(\mu,\nu)$. For each positive integer $j\leq k$, let $\mu^{j}$ denote the partition corresponding to the shape of the tableau formed by the entries not exceeding $j$ in $T$. Then it is clear that $\mu^{j}$ is obtained from $\mu^{j-1}$ by adding a horizontal strip of size $\nu_j$ for each $1\leq j\leq k$. Now fix $j$ and write $\rho=\mu^{j-1}$ and $\lambda=\mu^{j}$. The contribution to $s_q(T)$ of entries equal to $j$ in $T$ is the multinomial coefficient 
  \begin{align*}
    {|\lambda|-|\rho|\brack \lambda_1-\rho_1,\lambda_2-\rho_2,\ldots}_q.
  \end{align*}
  On the other hand, the contribution to $r_q(T)$ coming from entries equal to $j$ is given by
  \begin{align*}
    \prod_{i\geq 1}{\rho_i-\rho_{i+1} \brack \lambda_{i+1}-\rho_{i+1}}_q [\lambda_{i+1}-\rho_{i+1}]_q!
  \end{align*}
  Multiplying the last two expressions above, it follows that the total contribution of all entries equal to $j$ to the product $r_q(T)s_q(T)$ is given by
  \begin{align*}
    \theta_{\lambda/\rho}(q)=\frac{[|\lambda|-|\rho|]_q!}{[\lambda_1-\rho_1]_q!} \prod_{i\geq 1}{\rho_i-\rho_{i+1}\brack \lambda_{i+1}-\rho_{i+1}}_q.
  \end{align*}
Therefore, the contribution of all entries of $T$ to $r_q(T)s_q(T)$ is equal to
  \begin{align*}
 \prod_{j=1}^k \theta_{\mu^{j}/\mu^{j-1}}(q),
  \end{align*}
  and the proposition follows. 
\end{proof}
\begin{corollary}\label{cor:recforb}
Let $\mu,\nu$ be partitions of $n$ with $\nu=(\nu_1,\ldots,\nu_k).$ Then
\begin{align*}
  b_{\mu\nu}(q)=\sum_{\substack{\rho:\mu/\rho\text{ is a horizontal}\\ \text{strip of size } \nu_k}}b_{\rho,\hat{\nu}}(q) \; \theta_{\mu/\rho}(q),
\end{align*}
where $\hat{\nu}$ denotes the partition obtained by deleting the last part of $\nu$.
\end{corollary}
Corollary \ref{cor:recforb} can be used to efficiently compute the polynomials $b_{\mu\nu}(q)$ recursively, starting with the base case $b_{\mu\nu}(q)=1$ when $\mu$ and $\nu$ are empty partitions.
\begin{theorem}\label{thm:whittakerconnection}
If $W_{\mu}=\sum_{\nu}a_{\mu\nu}(q)m_\nu$ denotes the monomial expansion of the $q$-Whittaker function, then
  \begin{align*}
 b_{\mu\nu}(q)   =\frac{\prod_{i\geq 1}[\nu_i]_q!}{\prod_{i\geq 1}[\mu_i-\mu_{i+1}]_q!}  a_{\mu\nu}(q).
  \end{align*}
\end{theorem}
\begin{proof}
  Write $\nu=(\nu_1,\ldots,\nu_k)$ and consider a chain $\emptyset=\mu^{0}\subseteq \mu^{1}\subseteq \cdots\subseteq \mu^{k}=\mu$ of partitions such that $\mu^{j}/\mu^{j-1}$ is a horizontal $\nu_j$-strip for $1\leq j\leq k$. We have
  \begin{align*}
    \psi_{\mu^{j}/\mu^{j-1}}(q)&=\prod_{i\geq 1}{\mu^{j}_i-\mu^{j}_{i+1} \brack \mu^{j}_i-\mu^{j-1}_i}_q,
  \end{align*}
  and
\begin{align*}                                     
        \theta_{\mu^{j}/\mu^{j-1}}(q)&=\frac{[\nu_j]_q!}{[\mu^{j}_1-\mu^{j-1}_1]_q!} \prod_{i\geq 1}{\mu^{j-1}_i-\mu^{j-1}_{i+1}\brack \mu^{j}_{i+1}-\mu^{j-1}_{i+1}}_q. 
  \end{align*}
  It is easily verified that
  \begin{align*}
\prod_{j=1}^k \theta_{\mu^{j}/\mu^{j-1}}(q)= \frac{\prod_{i\geq 1}[\nu_i]_q!}{\prod_{i\geq 1}[\mu_i-\mu_{i+1}]_q!}  \prod_{j=1}^k \psi_{\mu^{j}/\mu^{j-1}}(q) .
  \end{align*}
The proposition now follows from the expressions for $a_{\mu\nu}(q)$ and $b_{\mu\nu}(q)$ in terms of $\psi_{\mu/\rho}(q)$ and $\theta_{\mu/\rho}(q)$.  
\end{proof}
We now consider an application of Theorem \ref{thm:whittakerconnection} to binary integer matrices. The coefficients $M_{\lambda\nu}$ in the expansion
\begin{align*}
  e_\nu=\sum_{\lambda}M_{\lambda\nu}m_\lambda
\end{align*}
correspond to the number of $(0,1)$-matrices whose row and column sums are given by $\lambda$ and $\nu$ respectively (Stanley \cite[Prop.\ 7.4.1]{MR1676282}). A formula for computing $M_{\lambda\nu}$ using Kostka numbers appears in Macdonald \cite[I.(6.7)]{MR1354144}. An explicit expression not involving Kostka numbers was found by Johnsen and Straume~\cite{MR0878703}. Theorem \ref{thm:whittakerconnection} together with the $q=1$ specialization of the $q$-Whittaker functions in Eq.\ \eqref{eq:qwspecialization}, implies the following corollary.   
\begin{corollary}\label{cor:binmat}
Let $\lambda$ and $\nu$ be integer partitions. The number of binary integer matrices with row sums $\lambda$ and column sums $\nu$ is given by
  \begin{align*}
    M_{\lambda \nu}=b_{\lambda'\nu}(1)\frac{\prod_{i\geq 1}m_i(\lambda)!}{\prod_{i\geq 1}\nu_i!},
  \end{align*}
  where $\lambda'$ denotes the partition conjugate to $\lambda$ and  $m_i(\lambda)$ is the multiplicity of the integer $i$ as a part of $\lambda$ for $i\geq 1$. 
\end{corollary}
  In view of Definition \ref{def:bmualpha}, Corollary \ref{cor:binmat} implies the following formula for $M_{\lambda\nu}$ as a sum over semistandard tableaux.
  \begin{corollary}\label{cor:binmatviassyt}
We have    
  \begin{align*}
    M_{\lambda\nu}=\frac{\prod_{i\geq 1}m_i(\lambda)!}{\prod_{i\geq 1}\nu_i!}\sum_{T\in \ssyt(\lambda',\nu)} \left(\prod_{i\geq 2}\prod_{j\geq 1}r_{ij}(T)\right) \prod_{i\geq 1}{\nu_i \choose \beta^i(T)},
  \end{align*}
  where $\beta^i(T)$ are as in Definition \ref{def:bit} and $r_{ij}(T)$ are as in Definition \ref{def:rqt}. 
\end{corollary}
The formula above provides a very efficient nonrecursive method to compute $M_{\lambda\nu}$ when the Kostka number $K_{\lambda'\nu}$ is small. 

We now consider an application of the above results to a classical combinatorial problem. Given a positive integer $m$, a \emph{chord diagram} on $2m$ points is a collection of $2m$ equally spaced points on the circumference of a circle joined by chords in pairs. Such chord diagrams are easily seen to correspond to set partitions of $[2m]$ in which all blocks are of size 2. The chord diagrams for $m=2$, are shown in Figure \ref{fig:chords}. 
\begin{figure}[!ht]
  \centering
  \includegraphics[scale=.5]{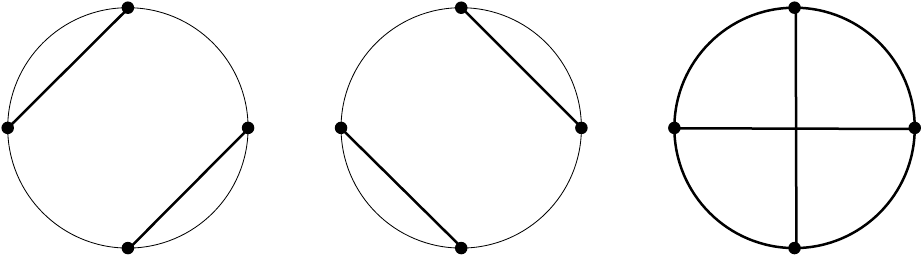}
  \caption{All chord diagrams on 4 points.} 
  \label{fig:chords}
\end{figure}
Among all chord diagrams on $2m$ points, let $t_{m,k}$ denote the number of chord diagrams which have $k$ pairs of crossing chords (these coincide with crossings of arcs in the corresponding set partition) and define $T_m(q)=\sum_{k\geq 0}t_{m,k}q^k$. Thus $T_m(q)$ is the generating polynomial for chord diagrams on $2m$ points by number of crossings. For instance, $T_2(q)=2+q$. The polynomials $T_m(q)$ were investigated by Touchard \cite{MR37815,MR0036006,MR46325}  in the context of the stamp folding problem. Building on Touchard's work, Riordan~\cite{MR366686} gave a beautiful analytic formula for $T_m(q)$:    
\begin{equation}
  \label{eq:touchard-riordan}
  T_m(q)(1-q)^m = \sum_{i=0}^m (-1)^i\left[\binom{2m}{m-i}-\binom{2m}{m-i-1}\right]q^{\binom{i+1}2}.
\end{equation}
 An alternative proof of this formula using ideas from the theory of continued fractions can be found in Read~\cite{MR556055}. A bijective proof due to Penaud~\cite{MR1336847} is also known. The reader is referred to Aigner \cite[p.\ 337]{MR2339282} for a very nice exposition.

We will prove that the polynomials $T_m(q)$ can be expressed in terms of the $q$-Whittaker functions. Recall that the ring of symmetric functions is endowed with the Hall scalar product $\langle \cdot,\cdot\rangle$ with respect to which the monomial and complete homogeneous symmetric functions are dual bases (Macdonald \cite[p. 63]{MR1354144}). In other words, we have $\langle m_\lambda,h_\mu\rangle=\delta_{\lambda\mu}$, where $\delta_{\lambda\mu}$ denotes the Kronecker delta function. The coefficients $a_{\mu\nu}(q)$ in the monomial expansion of $W_\mu$ can also be written in terms of the scalar product as $\langle W_\mu,h_{\nu}\rangle.$
\begin{theorem}\label{thm:touchardasscalar}
  We have
  \begin{align*}
    T_m(q)=\frac{1}{[m]_q!} \langle W_{(m,m)},h_1^{2m} \rangle.
  \end{align*}
\end{theorem}
\begin{proof}
By Theorem \ref{thm:whittakerconnection},
\begin{align*}
  b_{\mu\nu}(q)=\frac{\prod_{i\geq 1}[\nu_i]_q!}{\prod_{i\geq 1}[\mu_i-\mu_{i+1}]_q!}\langle W_\mu,h_{\nu}\rangle.
\end{align*}
Now take $\nu=(1^{2m})$ and $\mu=(m,m)$. It follows from Theorem \ref{thm:bviasetstat} that
\begin{align*}
  b_{\mu\nu}(q)=\sum_{\AA\in \Pi_{2m}(2^m)}q^{i_\alpha(\AA)}=T_m(q),
\end{align*}
since $i_\alpha(\AA)$ is simply equal to the number of crossings of $\AA$ by Remark \ref{rem:crossings}. It follows that the polynomials $T_m(q)$ can be expressed as 
\begin{align*}
  T_m(q)&=\frac{1}{[m]_q!} \langle W_{(m,m)},h_1^{2m} \rangle.\qedhere
\end{align*}
\end{proof}
\section{$q$-Stirling numbers indexed by integer partitions}
\label{sec:stirling}
In this section we define a new class of $q$-Stirling numbers $S_q(n,m;\nu)$ indexed by integer partitions $\nu$ of $n$. In the case where $\nu$ has all parts equal to 1, these numbers coincide with the $q$-Stirling numbers of the second kind defined by Carlitz \cite{MR1501675}. First, we consider a $q=1$ analog of these numbers. Given a composition $\alpha$ of $n$, let $S(n,m;\alpha)$ denote the number of set partitions $\AA$ of $[n]$ which contain $m$ blocks and intersect $\AA_\alpha$ minimally. It is easy to see that $S(n,m;\alpha)$ is invariant under a rearrangement of the coordinates of $\alpha$; therefore we only need to consider $S(n,m;\nu)$ for integer partitions $\nu$. Note that when $\nu$ has all parts equal to 1, $S(n,m;\nu)$ equals the number of set partitions of $[n]$ into $m$ blocks and is therefore simply the classical Stirling number of the second kind.  

Since $b_{\mu\nu}(1)=|\Pi(\mu',\nu)|$ by Theorem~\ref{thm:bviasetstat}, it follows that
\begin{align}\label{eq:stirlingnu}
  S(n,m;\nu)=\sum_{\substack{\mu\vdash n\\ \mu_1=m}}b_{\mu\nu}(1).
\end{align}
When $\nu$ has precisely one part, $S(n,m;\nu)$ is 1 if $n=m$ and zero otherwise. The numbers $S(n,m;\nu)$ satisfy the following recursion.
\begin{proposition}\label{prop:nustirling}
For each partition $\nu=(\nu_1,\ldots,\nu_k)$ of $n$,
  \begin{align*}
    S(n,m;\nu)=\sum_{r=0}^{\nu_k}S(n-\nu_k,m-r,\hat{\nu}){\nu_k \choose r}{m-r \choose \nu_k-r}(\nu_k-r)!,
  \end{align*}
  where $\hat{\nu}$ denotes the partition obtained by deleting the last part of $\nu$.
\end{proposition}
\begin{proof}
  Write $\AA_\nu=\cup_{i=1}^k A_i$ where $A_i=\{\nu_1+\cdots+\nu_{i-1}+1,\ldots,\nu_1+\cdots+\nu_i\}$ for $1\leq i\leq k$. Consider a set partition $\AA$ counted by $S(n,m;\nu).$ If we delete the elements of $A_k$ from the blocks of $\AA$, then we obtain a set partition $\AA'$ of $\cup_{i=1}^{k-1}A_i$ such that $\AA'$ intersects $\AA_{\hat{\nu}}$ minimally. The number of blocks in $\AA'$ is precisely $m-r$ where $r$ is the number of elements of $A_k$ that lie in singleton blocks of $\AA$. The number of set partitions $\AA'$ that can be obtained by the procedure above is counted by $S(n-\nu_k,m-r,\hat{\nu})$. The number of set partitions $\AA$ which yield a given set partition $\AA'$ can be counted by first choosing $r$ elements of $A_k$ which lie in singleton blocks and then assigning each of the remaining $\nu_k-r$ elements to the $m-r$ blocks of $\AA'$. Thus a given set partition $\AA'$ arises in precisely
  \begin{align*}
    {\nu_k \choose r}{m-r \choose \nu_k-r}(\nu_k-r)!
  \end{align*}
  ways. The proposition follows easily from these observations.
\end{proof}
For nonnegative integers $n,m$, the Carlitz $q$-Stirling numbers of the second kind are defined by
\begin{align}\label{eq:qsrec} 
  S_q(n,m)=S_q(n-1,m-1)+[m]_qS_q(n-1,m) \quad (n\geq 1,m\geq 1),
\end{align}
with $S_q(n,m)=\delta_{nm}$ (Kronecker delta) when either $n=0$ or $m=0$. The following definition  may be viewed as a $q$-analog of Eq.\ \eqref{eq:stirlingnu}. 
\begin{definition}\label{def:qstirlnu}
Define 
\begin{align*}
  S_q(n,m;\nu):=\sum_{\substack{\mu\vdash n\\\mu_1=m}}q^{\sum_{j\geq 2}{\mu_j \choose 2}}b_{\mu\nu}(q).
\end{align*}
\end{definition}

  When $\nu$ has all parts equal to 1, we will see that $S_q(n,m;\nu)$ coincides with the $q$-Stirling numbers of the second kind defined by Carlitz.     

  In view of Theorem \ref{thm:bviasetstat} which expresses $b_{\mu\nu}(q)$ in terms of a set partition statistic, we have the following alternate combinatorial expression for $S_q(n,m;\nu)$.
  \begin{proposition}\label{prop:stirlingviasetstat} 
    For each Mahonian statistic $\varphi$, 
    \begin{align*}
      S_q(n,m;\nu)=\sum q^{i_\nu(\AA)+\varphi_\nu(\AA)+\sum_{i\geq 1}(i-1)(\lambda^\AA_i-1)}, 
    \end{align*}
    where the sum is taken over all set partitions $\AA\in \Pi_n$ with $m$ blocks which intersect $\AA_\nu$ minimally and $\lambda^\AA$ denotes the shape of $\AA$. 
  \end{proposition}
The next result shows that an analog of Proposition \ref{prop:nustirling} holds for $S_q(n,m;\nu)$. This recursion will be used in Section \ref{sec:rooktheory} to show that $S_q(n,m;\nu)$ corresponds to $q$-rook numbers of suitably defined Ferrers boards.
\begin{theorem}\label{thm:stirlrec}
  For each partition $\nu=(\nu_1,\ldots,\nu_k)$ of $n$,
  \begin{align*}
 S_q(n,m;\nu)=\sum_{r=0}^{\nu_k}q^{\nu_k-r \choose 2}{\nu_k \brack r}_q{m-r \brack \nu_k-r}_q[\nu_k-r]_q!\; S_q(n-\nu_k,m-r,\hat{\nu}),
  \end{align*}
    where $\hat{\nu}$ denotes the partition obtained by deleting the last part of $\nu$.
\end{theorem}
\begin{proof}
If we write $n(\mu)=\sum_{j\geq 2}{\mu_j \choose 2}$, it follows from Proposition \ref{cor:recforb} that  
  \begin{align*}
    S_q(n,m;\nu)&=\sum_{\substack{\mu\vdash n\\\mu_1=m}}q^{n(\mu)}b_{\mu\nu}(q)\\
                &=\sum_{\substack{\mu\vdash n\\\mu_1=m}}q^{n(\mu)}\sum_{\substack{\rho:\mu/\rho\text{ is a horizontal}\\ \text{strip of size } \nu_k}}b_{\rho,\hat{\nu}}(q) \frac{[\nu_k]!_q}{[\mu_1-\rho_1]!_q} \prod_{j\geq 1}{\rho_{j}-\rho_{j+1} \brack \mu_{j+1}-\rho_{j+1}}_q\\
                &=\sum_{r=0}^{\nu_k}\sum_{\substack{\mu\vdash n\\\mu_1=m}}\;\sum_{\substack{\rho:\mu/\rho\text{ is a horizontal}\\ \text{strip of size } \nu_k \\\rho_1=m-r}}q^{n(\mu)}b_{\rho,\hat{\nu}}(q) \frac{[\nu_k]!_q}{[r]!_q} \prod_{j\geq 1}{\rho_{j}-\rho_{j+1} \brack \mu_{j+1}-\rho_{j+1}}_q.
  \end{align*}
  Now interchange the two inner sums to obtain 
  \begin{align*}
               &\sum_{r=0}^{\nu_k}\frac{[\nu_k]!_q}{[r]!_q}\sum_{\substack{\rho\vdash n-\nu_k\\ \rho_1=m-r}}\; \sum_{\substack{\mu:\mu/\rho\text{ is a horizontal}\\ \text{strip of size }\nu_k \\ \mu_1=m}}q^{n(\mu)}b_{\rho,\hat{\nu}}(q)  \prod_{j\geq 1}{\rho_{j}-\rho_{j+1} \brack \mu_{j+1}-\rho_{j+1}}_q\\
                &=\sum_{r=0}^{\nu_k}\frac{[\nu_k]!_q}{[r]!_q}\sum_{\substack{\rho\vdash n-\nu_k\\ \rho_1=m-r}}b_{\rho,\hat{\nu}}(q)q^{n(\rho)+{\nu_k-r \choose 2}}\sum_{\substack{\mu:\mu/\rho\text{ is a horizontal}\\ \text{strip of size } \nu_k\\ \mu_1=m}}q^{n(\mu)-n(\rho)-{\nu_k-r \choose 2}}  \prod_{j\geq 1}{\rho_{j}-\rho_{j+1} \brack \mu_{j+1}-\rho_{j+1}}_q.
  \end{align*}
  We claim that the innermost sum above is simply ${m-r \brack \nu_k-r}_q.$ To see this, suppose $\rho$ has $\ell$ parts and let $x_j=\mu_{j+1}-\rho_{j+1}$ for $1\leq j\leq \ell$. Then $n(\mu)-n(\rho)=\sum_{j\geq 1}{x_j \choose 2}+x_j\rho_j$. Since $\sum_{j\geq 1}x_j=\nu_k-r$, we also have ${\nu_k-r \choose 2}=\sum_{j\geq 1}{x_j \choose 2}+\sum_{1\leq j<i}x_jx_i$. Therefore, the innermost sum above is equal to
  \begin{align*}
   \sum_{x_1+x_2+\cdots+x_{\ell}=\nu_k-r} q^{\sum_{j=1}^{\ell}x_j(\rho_{j+1}-(x_{j+1}+\cdots+x_{\ell}))} \prod_{j=1}^{\ell}{\rho_{j}-\rho_{j+1} \brack x_{j}}_q={\rho_1 \brack \nu_k-r}_q,
  \end{align*}
by an $\ell$-fold Vandermonde convolution for $q$-binomial coefficients. The claim now follows since $\rho_1=m-r$. It follows that 
  \begin{align*}
    S_q(n,m;\nu)  &=\sum_{r=0}^{\nu_k}\frac{[\nu_k]!_q}{[r]!_q}{m-r \brack \nu_k-r}_qq^{\nu_k-r \choose 2}\sum_{\substack{\rho\vdash n-\nu_k\\ \rho_1=m-r}}b_{\rho,\hat{\nu}}(q)q^{n(\rho)}\\
    &=\sum_{r=0}^{\nu_k}\frac{[\nu_k]!_q}{[r]!_q}{m-r \brack \nu_k-r}_qq^{\nu_k-r \choose 2}S_q(n-\nu_k,m-r,\hat{\nu}), 
  \end{align*}
proving the theorem.
\end{proof}

\begin{remark}\label{rem:stirlingbase}
When $\nu$ has a single part, it is easily verified using Definition \ref{def:qstirlnu} that $S_q(n,m;\nu)=1$ if $n=m$ and 0 otherwise. Along with this base case, Theorem \ref{thm:stirlrec} can be used to efficiently compute $S_q(n,m;\nu)$.   
\end{remark}

Note that when $\nu$ has all parts equal to 1, the recursion of Theorem \ref{thm:stirlrec} reads
\begin{align*}
  S_q(n,m;(1^n))=S_q(n-1,m-1;(1^{n-1}))+[m]_qS_q(n-1,m;(1^{n-1})),
\end{align*}
which coincides with the defining recursion \eqref{eq:qsrec} for the $q$-Stirling numbers of Carlitz. 
\section{Stirling numbers and $q$-rook theory} 
\label{sec:rooktheory}
Rook theory was introduced by Kaplansky and Riordan \cite{MR0016082} in the 1940s and, ever since, has found several applications in enumerative combinatorics. The basic idea of rook theory is to count the number of placements of a given number of nonattacking rooks on a chess board; these are called rook numbers of the board. Garsia and Remmel \cite{MR834272} developed $q$-analogues of the rook numbers by defining a statistic on rook placements. Haglund \cite{MR1422066} found intricate connections between $q$-rook theory and basic hypergeometric series and proved that a large class of $q$-rook numbers can be expressed in terms of basic hypergeometric series of Karlsson--Minton type (see Gasper and Rahman \cite[Sec.\ 1.9]{MR2128719} for the definition). Ding \cite{MR1452940} found connections between rook theory and algebraic topology by exhibiting algebraic varieties for which the Poincaré polynomials of cohomology are in fact $q$-rook polynomials. Rook theory has also been considered in the context of matrices over finite fields; see, for instance, the papers by Haglund \cite{MR1612854} and Solomon \cite{MR1065211}. In this section we briefly describe the $q$-rook theory of Garsia and Remmel and show that the $q$-Stirling numbers $S_q(n,m;\nu)$ discussed in Section \ref{sec:stirling} are closely related to the $q$-rook numbers of certain types of Ferrers boards.   

For a sequence $ a_1\leq \cdots \leq a_r$ of nonnegative integers, the Ferrers board of shape $(a_1,\ldots,a_r)$ is defined by
\begin{align*}
  B(a_1,\ldots,a_r)=\{(i,j):1\leq i\leq r, 1\leq j\leq a_i\}.
\end{align*}
 The board $B$ is visually represented as an array of squares in the Cartesian plane such that the square corresponding to $(i,j)\in B$ has its upper right corner at the point $(i,j)$. For instance, the board $B(2,4,5)$ is show in Figure \ref{fig:fb245}.
\begin{figure}[!ht]
  \centering
 % \ytableaushort{\none\none \ ,\none \ \ , \none \ \ ,\ \ \ ,\ \ \  }\qquad
  % \ytableaushort{\none\none \none ,\none\none \none\ ,\none \none \ \ , \none \ \ \ ,\ \ \ \ }
  \includegraphics[scale=1]{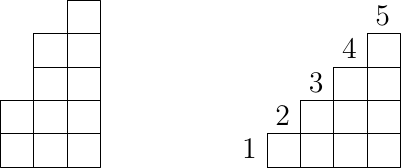}
  \caption{Left: the Ferrers board $B(2,4,5)$; Right: the staircase board of size 5 with rows and columns labeled.}
  \label{fig:fb245}
\end{figure}
By the staircase board of size $n$, we mean the Ferrers board $B_n=B(1,2,\ldots,n-1)$. Given a Ferrers board $B$, a placement of $m$ nonattacking rooks on $B$ is simply a subset $P\subseteq B$ of cardinality $m$ such that, for distinct elements $(i,j)$ and $(k,l)$ of $P$, we have $i\neq k$ and $j\neq l$. Each set partition $\AA$ of $[n]$ with $m$ blocks corresponds uniquely to a placement of $n-m$ nonattacking rooks on a staircase board of size $n$ (Stanley \cite[Sec. 2.4]{MR2868112}). To see this correspondence, number the rows of the staircase board from $1$ to $n-1$ from bottom to top. The columns are numbered from $2$ to $n$ from left to right as shown in Figure \ref{fig:fb245}. A rook (indicated by a cross) is placed in the box in row $i$ and column $j$ whenever $i$ and $j$ are consecutive elements in the same block of $\AA$. For instance, the rook placement corresponding to the set partition $\AA=14|235$ is shown in Figure \ref{fig:rooks}.
\begin{figure}[!ht]
  \centering
  \includegraphics[scale=1]{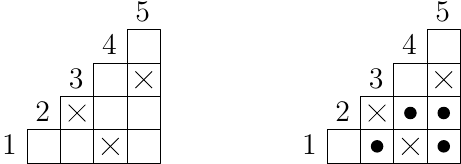}  
  \caption{Left: rook placement for $\AA=14|235$. Right: $\rho(P)=3$.}
  \label{fig:rooks}
\end{figure}
 For a rook placement $P$, we mark all squares that are below or to the right of some rook by a thick dot; these squares are said to be \emph{cancelled} by the rooks. The number of cells which do not contain a rook or a dot is denoted $\rho(P)$. An example is shown in Figure~\ref{fig:rooks}. The following definition is due to Garsia and Remmel \cite{MR834272}. 
\begin{definition}
The $m$th $q$-rook number of the Ferrers board $B$ is defined by
\begin{align*}
  r_m(B):=\sum_{P} q^{\rho(P)},
\end{align*}
where the sum is taken over all placements $P$ of $m$ nonattacking rooks on $B$.  
\end{definition}
  For each set partition $\AA$, define $\rho(\AA)=\rho(P)$, where $P$ is the placement of rooks on the staircase board corresponding to $\AA$. The $q$-Stirling numbers of the second kind defined in Section~\ref{sec:stirling} can be expressed in terms of the rook numbers of the staircase board $B_n$ by $S_q(n,m)=q^{-{m \choose 2}}r_{n-m}(B_n)$ (this is essentially \cite[Eq.\ 3.10]{MR834272} with a slightly different normalization for $S_q(n,m)$). Therefore    
\begin{align*}
  S_q(n,m)=\sum_{\AA}q^{\rho(\AA)-{m \choose 2}},
\end{align*}
where the sum is taken over all set partitions $\AA$ of $[n]$ with $m$ blocks.
\begin{definition}\label{def:stilde}
  Given an integer partition $\nu$, define
  \begin{align*}
    \tilde{S}_q(n,m;\nu)=\sum_{\AA}q^{\rho(\AA)-{m \choose 2}},
  \end{align*}
where the sum is taken over all set partitions $\AA$ of $[n]$ with $m$ blocks such that $\AA$ intersects $\AA_\nu$ minimally.
\end{definition}
When $\nu=(1^n)$, it is clear that $\tilde{S}_q(n,m;\nu)=S_q(n,m),$ the classical $q$-Stirling number of Carlitz. We will prove that $\tilde{S}_q(n,m;\nu)$ is in fact equal to $S_q(n,m;\nu)$ which was defined in terms of the polynomials $b_{\mu\nu}(q)$ in Section \ref{sec:stirling}.

\begin{definition}
  For each partition $\nu$, let $B(\nu)$ denote the Ferrers board whose column lengths are successively given by
  $$
\underbrace{\nu_1,\ldots,\nu_1}_{\text{$\nu_2$ times}},\underbrace{\nu_1+\nu_2,\ldots,\nu_1+\nu_2}_{\text{$\nu_3$ times}},\ldots,\underbrace{\nu_1+\cdots+\nu_{k-1},\ldots,\nu_1+\cdots+\nu_{k-1}}_{\text{$\nu_k$ times}}.   $$
\end{definition}
\begin{lemma}\label{lem:stirlingasrook} 
  We have
  \begin{align*}
\tilde{S}_q(n,m;\nu)=  q^{-{m \choose 2}+\sum_{j\geq 1}{\nu_j \choose 2}}  r_{n-m}(B(\nu)).
  \end{align*}
\end{lemma}
\begin{proof}
Note that the set partitions $\AA$ which contribute to $\tilde{S}_q(n,m;\nu)$ correspond to placements of $n-m$ nonattacking rooks on $B_n$ such that no rook is placed in row $i$ and column $j$ whenever $i$ and $j$ lie in the same block of $\AA_\nu$. For instance, when $\nu=(3,2,4)$, this corresponds to all placements of nonattacking rooks on the unmarked squares of the board shown in Figure \ref{fig:forbid}.

\begin{figure}[!ht]
  \centering
  \includegraphics[scale=1]{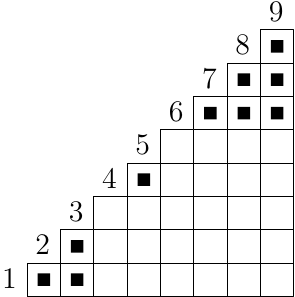} 
  \caption{The forbidden squares for rooks when $\nu=(3,2,4)$.}
  \label{fig:forbid} 
\end{figure}
 The forbidden squares for rooks form smaller staircase boards of size $\nu_i$ for $i\geq 1$. For each partition $\nu$ of $n$, deleting the forbidden squares from a staircase board of size $n$ results in the Ferrers board $B(\nu)$. Since $B(\nu)$ is obtained from the staircase board of size $n$ by deleting precisely $\sum_{j\geq 1}{\nu_j \choose 2}$ squares, it follows that  
  \begin{align*}
\tilde{S}_q(n,m;\nu)&=  q^{-{m \choose 2}+\sum_{j\geq 1}{\nu_j \choose 2}}  r_{n-m}(B(\nu)).\qedhere
  \end{align*}
\end{proof}
Denote by $B(a^b)$ the rectangular Ferrers board which has $b$ columns, each of length $a$. The following Lemma is a slight generalization of \cite[Thm.\ 3.3]{MR834272} and is easily proved by the same combinatorial argument given there.
\begin{lemma}\label{lem:rectangular}
  The $m$th $q$-rook number of the rectangular board $B=B(a^b)$ is given by
  \begin{align*}
    r_m(B)=q^{(a-m)(b-m)}{b \brack m}_q \frac{[a]_q!}{[a-m]_q!}.
  \end{align*}
\end{lemma}
The next lemma may be viewed as a generalization of \cite[Thm.\ 1.1]{MR834272}.
\begin{lemma}\label{lem:convolution}
Let $B$ be a Ferrers board in which each column has height at most $a$. If $\bar{B}=B\cup B(a^b)$ denotes the Ferrers board obtained by adding $b$ columns of height $a$ to the right of $B$, then
  \begin{align*}
    r_m(\bar{B})&=\sum_{j=0}^m r_j(B) q^{(a-m)(b-m+j)}{b \brack m-j}_q \frac{[a-j]_q!}{[a-m]_q!}.
  \end{align*}
\end{lemma}
\begin{proof}
  Each placement of $m$ rooks on $\bar{B}$ corresponds to a placement of $j$ rooks on $B$ and $m-j$ rooks on $B(a^b)$ for some $0\leq j\leq m$. Placing $j$ nonattacking rooks on $B$ cancels $j$ rows in $B(a^b)$, leaving $a-j$ rows for the remaining $m-j$ rooks. Therefore,
  \begin{align*}
     r_m(\bar{B})&=\sum_{j=0}^m r_j(B)\; r_{m-j}(B((a-j)^b)).
  \end{align*}
The result now follows from Lemma \ref{lem:rectangular}.
\end{proof}
  \begin{theorem}
    We have $S_q(n,m;\nu)=\tilde{S}_q(n,m;\nu).$
  \end{theorem}
  \begin{proof}
When $\nu$ has a single part, it is easily verified using Definition \ref{def:stilde} that $\tilde{S}_q(n,m;\nu)=1$ if $n=m$ and 0 otherwise. By Remark \ref{rem:stirlingbase}, it follows that $S_q(n,m;\nu)$ and $\tilde{S}_q(n,m;\nu)$ agree when $\nu$ has a single part. Therefore, it suffices to prove that $\tilde{S}_q(n,m;\nu)$ satisfies the recurrence of Theorem \ref{thm:stirlrec}. If $\hat{\nu}$ denotes the partition obtained by deleting the last part of $\nu$, then
  \begin{align*}
    B(\nu)=B(\hat{\nu})\cup B(a^b),
  \end{align*}
  where $a=\nu_1+\cdots+\nu_{k-1}=n-\nu_k$ and $b=\nu_k$. 
 Set $B=B(\hat{\nu})$ in Lemma \ref{lem:convolution} to obtain 
  \begin{align*}
    r_m(B(\nu))=\sum_{j}r_j(B(\hat{\nu}))q^{(n-\nu_k-m)(\nu_k-m+j)}{\nu_k \brack m-j}_q\frac{[n-\nu_k-j]_q!}{[n-\nu_k-m]_q!}.
  \end{align*}
  Replacing $m$ by $n-m$, this becomes
  \begin{align*}
    r_{n-m}(B(\nu))&=\sum_{j}r_j(B(\hat{\nu}))q^{(m-\nu_k)(\nu_k-n+m+j)}{\nu_k \brack n-m-j}_q\frac{[n-\nu_k-j]_q!}{[m-\nu_k]_q!}\\
    &=\sum_{j}r_{n-\nu_k-m+j}(B(\hat{\nu}))q^{(m-\nu_k)j}{\nu_k \brack j}_q\frac{[m-j]_q!}{[m-\nu_k]_q!},
  \end{align*}
  where the last step follows by replacing $j$ by $n-\nu_k-m+j$. By Lemma~\ref{lem:stirlingasrook}, it follows that
  \begin{align*}
    q^{{m \choose 2}-\sum_{i\geq 1}{\nu_i \choose 2}}\tilde{S}_q(n,m;\nu)=\sum_{j}\tilde{S}_q(n-\nu_k,m-j;\hat{\nu})q^{{m-j \choose 2}-\sum_{i=1}^{k-1}{\nu_i \choose 2}} q^{(m-\nu_k)j}{\nu_k \brack j}_q\frac{[m-j]_q!}{[m-\nu_k]_q!},
  \end{align*}
  which simplifies to
  \begin{align*}
    \tilde{S}_q(n,m;\nu)=\sum_{j}\tilde{S}_q(n-\nu_k,m-j;\hat{\nu})q^{\nu_k-j \choose 2}{\nu_k \brack j}_q {m-j \brack \nu_k-j}_q[\nu_k-j]_q!. 
  \end{align*}
This recurrence is identical to the one proved for $S_q(n,m;\nu)$ in Theorem~\ref{thm:stirlrec}. 
\end{proof}

  \section{Acknowledgments} 
The first author thanks Per Alexandersson for some helpful discussions. The first author was partially supported by a MATRICS grant MTR/2017/000794 awarded by the Science and Engineering Research Board and an Indo-Russian project DST/INT/RUS/RSF/P41/2021. The second author's research was partly supported by FWF Austrian Science Fund grant P32305.      

\printbibliography  
\end{document}